\documentclass[12pt]{amsart}
\usepackage{amssymb}
\usepackage{amsmath}
\usepackage{epsfig}
\usepackage{ulem}
\usepackage{graphics}
\usepackage{hyperref}
\usepackage{tikz-cd}
\usepackage{enumerate,amsmath,enumitem}

\newcommand{\RR}{{\mathbb R}}
\newcommand{\CC}{{\mathbb C}}
\newcommand{\NN}{{\mathbb N}}
\newcommand{\ZZ}{\mathbb Z}
\newcommand{\QQ}{{\mathbb Q}}

\newcommand{\bS}{{\mathbb S}}
\newcommand{\TT}{{\mathbb T}}

\def\B{{\mathcal B}}

\def\L{{\mathcal L}}

\def\P{{\mathcal P}}
\def\R{{\mathcal R}}

\def\V{{\mathcal V}}

\def\dim{{\rm dim \:   }}

\def\F{{\tilde{\Phi}}}  
\def\Fetoile{{{\Phi}}}
\def\H{{\Psi}}
\def\G{{\Upsilon}}

\numberwithin{equation}{section}

\newtheorem{theo}{Theorem}
\newtheorem{prop}[theo]{Proposition}
\newtheorem{coro}[theo]{Corollary}
\newtheorem{lemma}[theo]{Lemma}
\newtheorem{defi}[theo]{Definition}
\newtheorem{exemple}[theo]{Example}

\theoremstyle{remark}

\begin{document}

\title{Conjugacy of Unimodular Pisot  Substitution Subshifts to Domain Exchanges}

\author[{F.} {Durand}]{ {Fabien} {Durand}}
\email{fabien.durand@u-picardie.fr}

\author[{S.} {Petite}]{ {Samuel} {Petite}}
\email{samuel.petite@u-picardie.fr}
\address{Laboratoire Ami\'enois
de Math\'ematiques Fondamentales et Appliqu\'ees, CNRS-UMR 7352,
Universit\'{e} de Picardie Jules Verne, 33 rue Saint Leu, 80000
Amiens, France.}

\thanks{Both authors acknowledge the ANR Program ANR Project IZES ANR-22-CE40-0011.}


\begin{abstract}
We prove that any unimodular Pisot substitution subshift  is measurably conjugate to  a domain exchange in  an Euclidean space  which { is a finite topological extension of a translation on a torus.}
This generalizes the pioneer works of Rauzy and Arnoux-Ito providing geometric realizations to any unimodular  Pisot substitution without any additional combinatorial condition.  
\end{abstract}

\date{}
\maketitle \markboth{Fabien Durand, Samuel Petite}{Unimodular Pisot  Substitutions and Domain Exchanges}

\section{Introduction}

A classical way to tackle problems of diffeomorphism dynamics is to code the orbits of the points through a well-chosen finite partition to obtain a "nice" subshift  which is easier to study (see the emblematic works \cite{Hadamard:1898} and \cite{Morse:1921}).
The interesting dynamical properties of the subshift  may translate into relevant properties of the original dynamical system. 

In the seminal paper \cite{Rauzy:1982}, G. Rauzy proposed to go in the other way round: take your favorite subshift and try to give it a geometrical representation.
He took what is now called the Tribonacci substitution given by 
$$
1 \mapsto 12, 2 \mapsto 13 \hbox{ and } 3 \mapsto 1,
$$

\noindent and proved that the subshift it generates is measure theoretically conjugate to a rotation on  the torus $\mathbb{T}^2$.
A similar result was already known for substitutions of constant length under some necessary and sufficient conditions \cite{Dekking:1978}.
Later, in \cite{Arnoux&Rauzy:1991},  the authors showed that subshifts whose block complexity is $2n+1$, and satisfy some technical conditions, are measure theoretically conjugate to an interval exchange on 6 intervals on the circle.
This family of subshifts includes the subshift generated by the Tribonacci substitution.

The Tribonacci substitution   has the specificity to be a unimodular and Pisot substitution, that is, its incidence matrix has determinant 1, its characteristic polynomial is irreducible and its dominant eigenvalue is a Pisot number (all its algebraic conjugates are, in modulus, strictly less than $1$).
These properties provide key arguments to prove the main result in \cite{Rauzy:1982} mentioned above.   
It  naturally leads to what is now called the {\bf Pisot conjecture} for symbolic dynamics: 

\medskip 

{\em Let $\sigma$ be a Pisot substitution.
Then, the subshift it generates has purely discrete spectrum, {i.e.}, it is measure theoretically conjugate to a group translation. }
\medskip

Many attempts have been done to solve this conjecture.
The usual  strategy is the same as  Rauzy's in \cite{Rauzy:1982}.
and also applies to the non-unimodular cases \cite{Siegel:2003,Sing:2006}: show first that the substitution subshift is measurably conjugate to a domain exchange (see Definition \ref{def:domainexchange}), then that this system is measurably conjugate to a  rotation on  a torus. 

In a widely cited, but unpublished, manuscript \cite{Host:1992} (see also Section 6.3.3 in \cite{Queffelec:2010}), B. Host proved that the Pisot conjecture is true for unimodular substitutions defined on two letters, provided a condition  called {\em strong coincidence condition} holds. 
This combinatorial condition first appeared in \cite{Dekking:1978}. 
Then, it was shown in \cite{Barge&Diamond:2002} that this condition is satisfied for any unimodular Pisot substitution on two letters. 
So the Pisot conjecture is true in this case \cite{Hollander&Solomyak:2003}.

Following the  Rauzy's strategy, but in a  different way than the Host's approach, in \cite{Arnoux&Ito:2001} the authors associate to any unimodular Pisot substitution a self-affine domain exchange called {\em Rauzy fractal}. They proved, this domain exchange is measurably conjugate to the substitution subshift provided the substitution satisfies  a combinatorial condition.   Few time later, Host's results were  generalized  by V. Canterini and A. Siegel in  \cite{Canterini&Siegel:2001}  to any unimodular Pisot substitution   and to the non-unimodular case \cite{Siegel:2003, Siegel:2004},  but without avoiding  the strong coincidence condition. 

The group translation is explicit in function of the incidence matrix of the substitution. It is  known that it corresponds to its maximal equicontinuous factor that also is its Kronecker factor  \cite{Barge&Kwapisz:2006}.
These works  led to the development of a huge number of techniques to study the  Rauzy fractals (see for instance  \cite{Fogg:2002} and references therein).  
Let us also mention  other fruitful  geometrical approaches by using tilings in  \cite{Barge&Kwapisz:2006, Baker&Barge&Kwapisz:2006} and more recently   in \cite{Barge:2016,Barge:2018} for the one-dimensional case. For more references on the topic, we refer to the  survey \cite{ABBLS:2015}.

The reducible case has also been investigated. It is known the conjecture is false in this setting (see \cite{EiIto:2005} and  Example \ref{ex:Sing}).
Nevertheless some examples have discrete spectrum \cite{Baker&Barge&Kwapisz:2006}. 

More recently a remarkable result has been obtained by M. Barge \cite{Barge:2018}: the Pisot conjecture is true for $\beta$-substitutions.

\medskip 

In this paper we generalize the Arnoux-Ito theorem \cite{Arnoux&Ito:2001}, also proven in \cite{Canterini&Siegel:2001}.
\begin{theo}
\label{theo:main}
 Let $\sigma $ be a unimodular Pisot substitution on $p+1$ letters.
Then  the  substitution subshift associated to $\sigma$  is measurably conjugate to an exchange
of domain by translation on the plane $\mathbb{R}^{p}$ through a continuous map. 
Moreover, it continuously factorizes via an a.s. constant-to-one map, onto a translation on the torus $\mathbb{T}^{p}$.
\end{theo}

The toral rotation is explicitly described in \cite{Canterini&Siegel:2001} (see also \cite{Fogg:2002}).
Let us point out that unlike all the previous works about the Pisot conjecture, no combinatorial condition (like the {\em strong coincidence property}) is required  to apply Theorem \ref{theo:main}.  In this paper we ignore this combinatorial condition showing this is not necessary to ensure the  measure theoretical conjugacy with a domain  exchange.
It provides a geometric realization, in terms of domain exchange, to any unimodular Pisot substitution subshift.
Notice the domain exchange may, a priori,  be different from the usual Rauzy fractal. 
To prove the Pisot conjecture, one still have to prove this domain exchange is measurably conjugate to the toral rotation.

\subsection{Comments on the proof of Theorem \ref{theo:main}} 
\label{sec:comments}
We prove  Theorem \ref{theo:main} as a corollary of a more general one  (Theorem \ref{theo:main2}, Subsection \ref{sec:proof}) that covers  a broader family of substitutions including some reducible Pisot substitutions.  
Let us make some comments on the proof of Theorem \ref{theo:main2}.  
Figure \ref{fig:blabla} might be useful to follow our strategy.

The canonical approach, coming from   Rauzy's seminal article \cite{Rauzy:1982}, is  related to the combinatorial properties of the subshift $\Omega(\sigma)$ associated to the substitution $\sigma$. A  map $\Phi_{Rauzy} \colon \Omega(\sigma) \to \RR^{p}$  is defined through the  prefix-suffix  automaton for irreducible unimodular Pisot substitutions \cite{Canterini&Siegel:2001}. Under coincidence hypotheses,  the applications $\Phi_{Rauzy}$ and $\pi \circ \Phi_{Rauzy}$, where $\pi \colon \RR^{p}  \to \TT^{p}$ denotes the canonical projection, are measurable conjugacy maps onto a domain exchange and  a minimal rotation.  
The key properties, to prove the measurable conjugacy  are then:
 The combinatorial property of coincidence;  The exchange by the transformation $\Phi_{Rauzy}$   of the shift transformation with a piecewise translation;  The  exchange by the same map of  the substitution transformation with the linear action of its incidence matrix. 

We overcome the  combinatorial coincidence condition by considering a conjugate subshift given by a proper substitution $\xi$ (that is a substitution mapping the letters to words starting with the same letter and ending with the same letter, see definition in Section \ref{def:subproper}), hence satisfying a form of coincidence.
 But the price to pay is to have to consider substitutions with a broader incidence matrix spectrum. In addition to the algebraic conjugates of the Pisot number it may include 0 and roots of unity as  eigenvalues. We call such substitutions    weakly irreducible Pisot substitutions (see Definition  \ref{def:Pisotsub}). 
It is not surprising because  being Pisot is not invariant under conjugacy, unlike being weakly irreducible Pisot.
In this context, the original strategy of  Rauzy does not work since it exist weakly irreducible Pisot substitutions defining  weakly-mixing subshifts (see Example \ref{ex:Solomyak} and discussion in Section \ref{sec:discuss}). 

Nevertheless the proper substitution allows to use Bratteli-Vershik representations.
They are  analogous to  prefix-suffix decompositions but take place into a general framework beyond substitution subshifts. 
Approximations of the eigenfunctions  through the return maps defining our Bratteli-Vershik representations (Proposition \ref{prop:condcont}), gives   a  canonical  continuous map $\tilde{\Phi} \colon \Omega(\xi) \to \RR^p$  
such that  $\pi \circ \tilde{\Phi}$  commutes the shift action  with the expected rotation  $(\mathbb{T}^{p} , x \mapsto x+ \alpha)$ (Lemma \ref{lemma:deltaV}).
Unfortunately, unlike $\Phi_{Rauzy}$,   the map $\F$ may not  exchange the substitution action with the one of its incidence matrix.  
In \cite{Host:1986,Canterini&Siegel:2001} it is a key argument to prove the conjugacy of the Pisot subshift with the domain exchange.

\begin{center}
\begin{figure}[h]
\begin{tikzcd}
(\Omega (\xi ) , \mu) \arrow[rr,"\tilde{\Phi}"] \arrow[dr,"\Phi"] && \tilde{\Phi } (\Omega (\xi )) \subset \RR^{p} \arrow[rr,"\pi"] & &(\mathbb{T}^{p} ,  \lambda_{\TT^{p}})\\
                                                        & (E , \lambda_{E}) \arrow[dashed,ur,"\Psi"] \arrow[dashed, rrru,"\pi \circ \Psi"]&&&
\end{tikzcd}
\caption{Sketch of the maps involved in the proof of Theorem \ref{theo:main2}: black arrows correspond to continuous maps and dashed arrows to almost everywhere defined maps.}
\label{fig:blabla}
\end{figure}
\end{center}

A modification of $\tilde{\Phi}$ by a piecewise translation overcomes this problem by giving a continuous  map $\Fetoile   \colon \Omega(\xi) \to E \subset  \RR^p$ having this commutation property (Lemma \ref{lemma:deltaV}). 
Then, the same strategy as in \cite{Canterini&Siegel:2001} enables to prove that $\Fetoile$
maps  by pullback   the normalized Lebesgue measure $\lambda_{E} = \frac 1{\lambda(E)}\lambda$ on $E$ to  the $S$-invariant measure $\mu$ (Lemma \ref{lemma:perron}).  Moreover it realizes a measurable conjugacy between  the subshift and a domain exchange $T$ on $E \subset \RR^{p}$ (Proposition \ref{prop:Finj} and proof of Theorem \ref{theo:main2}).  

However, in our  framework, due to the possible existence of a non-trivial kernel of the incidence matrix,
moding out by $\ZZ^p$ may not provide a factor onto a toral rotation,
that is, the map $\pi \circ \Fetoile$ may not commute with a rotation.
This complicates the proof. 
To solve this problem, we control the commutation fault by a function $\G$ (see \eqref{eq:defPhiGamma}) that enables us to extract a  piecewise translation $\H$ so that $\H \circ \Fetoile = \F$ a.e. (Equation \ref{eq:relPhiGamma}  gives the relation between these maps). 
We use then the commutation property of $\F$ to get the map  $\pi\circ \H \colon (E,T) \to (\TT^{p}, x \mapsto x+\alpha)$ is  an a.e. constant-to-one factor map. 
We deduce finally the map $\pi \circ \F \colon (\Omega(\xi), S) \to (\TT^{p}, x \mapsto x+\alpha)$ is an a.e. constant-to-one factor map   (see proof of Theorem \ref{theo:main2}). 

 As noticed by P. Mercat, in general the map $\psi$ is not trivial. For instance the usual Rauzy fractal associated to  the Pisot substitution $1\mapsto  1213$, $2\mapsto 121$, $3\mapsto 21$ is not a fundamental domain of the torus. Roughly speaking, the map $\psi$ translates some part of the Rauzy fractal to get a fundamental domain.  

It is important to notice that the measurable  conjugacy  with the domain exchange does not imply the measurable conjugacy with a rotation as illustrated by Example \ref{ex:Sing}. 

\subsection{Organisation of the paper}  We postpone to the next section the  basic definitions and notions we use for dynamical systems, substitution subshifts and Pisot substitutions. In Section \ref{sec:returnsub}, we prove by using the notion of return words, that any substitution subshift is conjugate to a { proper} substitution subshift.
We then show, in  Section \ref{sec:domex},  that such a subshift,  having  enough  multiplicatively independent eigenvalues (see Hypothesis \eqref{hypoii}), is  measurably conjugate to a self-affine domain exchange (Theorem \ref{theo:main2}).  A byproduct of these two results gives  Theorem \ref{theo:main}. 
The proof follows the same strategy as in \cite{Canterini&Siegel:2001}. 

We would like to emphasize that the irreducibility of unimodular Pisot substitutions  implies that the number of multiplicatively independent non trivial eigenvalues equals  $\sum_{0< \vert \lambda \vert <1} \textrm{ dim }E_{\lambda}$, where $E_{\lambda}$ denotes the eigenspace  associated with the eigenvalue $\lambda$ of  the substitution matrix.
This property is crucial in our proof and might be defined not only for substitution subshifts.
This suggests a possible extension of our main result to linearly recurrent subshifts, or $S$-adic subshifts, like in \cite{Berthe&Jolivet&Siegel:2012}.

 \subsection*{Acknowledgment} We thank B. Solomyak for providing us the  nice Example  \ref{ex:Solomyak} and P. Mercat for useful comments. We thank the anonymous reviewers for their careful reading  and their many insightful comments and suggestions.

\section{Basic definitions}\label{sec:basicdef}

\subsection{Words and sequences}
 
An {\it alphabet} $A$ is a finite set of elements called {\it
  letters}.
Its cardinality is $\#A$.
A {\it word} over $A$ is an element of the free monoid
generated by $A$, denoted by $A^*$. 
Let $x = x_0x_1 \cdots x_{n-1}$
(with $x_i\in A$, $0\leq i\leq n-1$) be a word, its {\it length} is
$n$ and is denoted by $|x|$. 
The {\it empty word} is denoted by $\epsilon$, $|\epsilon| = 0$. 
The set of nonempty words over $A$ is denoted by $A^+$. 
The elements of $A^{\mathbb{Z}}$ are called {\it sequences}. 
If $x=\cdots x_{-1} x_0 x_1\cdots$ is a sequence (with $x_i\in A$, $i\in \mathbb{Z}$) and $I=[k,l]$ an interval of
$\mathbb{Z}$ we set $x_I = x_k x_{k+1}\cdots x_{l}$ and we say that $x_{I}$
is a {\it factor} of $x$.  
The set of factors of length $n$ of $x$ is written
$\mathcal{L}_n(x)$ and the set of factors of $x$, or the {\it language} of $x$,
is denoted by $\mathcal{L}(x)$. 
The {\it occurrences} in $x$ of a word $u$ are the
integers $i$ such that $x_{[i,i + |u| - 1]}= u$. 
If $u$ has an occurrence in $x$, we also say that $u$ {\em appears} in $x$.
When $x$ is a word or an element of $A^\NN$ (called {\em one-sided sequence}),
we use the same terminology with similar notations.
The words $x_{[0,n]}$ are called {\em prefixes} of $x$.

Let $x$ be an element of $A^\ZZ$ or $A^\NN$.
A word $u$ is {\em recurrent} in $x$ if it appears in $x$ infinitely many times.
A sequence $x$ is {\em uniformly recurrent} if all $u\in \mathcal{L}(x)$ is recurrent in $x$  and the difference of two consecutive occurrences of $u$ in $x$ is bounded.

\subsection{Morphisms and matrices} 
Let $A$ and $B$ be two finite  alphabets. Let $\sigma$ be a {\it morphism} 
from $A^*$ to $B^*$. 
When $\sigma (A) = B$, we say $\sigma$ is a {\em coding}. 
If $\sigma (A)$ is included in $B^+$, we say $\sigma$ is
{\it non erasing}.
In this case, it induces by concatenation two maps also denoted $\sigma$:
One from $A^\NN$ to $B^\NN$ and the other from $A^{\mathbb{Z}}$ to $B^{\mathbb{Z}}$ satisfying $\sigma (x) = y$ with $\sigma (x_{[0,+\infty )}) = y_{[0,+\infty )}$ and $\sigma (x_{(-\infty  , 0]}) = y_{(-\infty , 0]}$. 

We naturally associate to the morphism $\sigma$ an {\em incidence matrix} $M_{\sigma} =
(m_{i,j})_{i\in B , j \in A }$, where $m_{i,j}$ is the number of
occurrences of $i$ in the word $\sigma(j)$. 
Notice that when $\sigma $ is an endomorphism, for any positive integer $n$ we get $M_{\sigma^n} = M_{\sigma}^n$.

We say that an endomorphism  is {\em primitive} whenever its incidence matrix is primitive ({i.e.}, when it has a power with strictly positive coefficients).  
The Perron's theorem tells that the dominant eigenvalue is a real simple root of the characteristic polynomial and is strictly greater than the modulus of any other eigenvalue. 
 Such eigenvalue with its eigenvectors are called {\em Perron  number} and {\em Perron eigenvectors}.

\subsection{Substitutions}\label{def:subproper}

Let us recall the definition of {\it substitution} as it appears in \cite{Queffelec:2010}: It is a non erasing
 endomorphism.

A substitution $\sigma$ is {\em left  proper} (resp. {\em right proper})  if all words $\sigma (b)$, $b \in A$, starts (resp. ends) with the same letter. 
Notice that all word  $\sigma^n(b)$, $b \in A$ for every power $n\ge 1$, starts (resp. ends) also with the same letter. 
For short, we say that a  left and right proper substitution is {\em proper}.

The {\em language} of $\sigma : A^* \to A^*$, denoted by $\L (\sigma )$, is the set of words having an occurrence in $\sigma^n (b)$ for some $n\in \mathbb{N}$ and $b\in A$. Notice that when $\sigma $ is primitive, then $\L(\sigma^n) = \L(\sigma)$ for any positive integer $n$.  

A {\em fixed point} of $\sigma $ is an element $x$ of $A^\ZZ$ such that $\sigma (x) = x$.
We say it is {\em admissible} whenever $\L (x)$ is a subset of $\L (\sigma )$. 
It could happen that a fixed point is not admissible as shown by the substitution defined by $1\mapsto 121$ and $2\mapsto 212$ and its non-admissible fixed point $x=\cdots 21212 . 21212 \cdots$.
It could also happen that $\sigma $ has no fixed point.
Nevertheless, if $\sigma$ is primitive, there always exists some $p>0$ such that $\sigma^p$ has an admissible fixed point. 
Indeed, there exist some $p>0$ and letters $l,r$ such that 

\begin{enumerate}
\item
$r$ is the last letter of $\sigma^p (r)$;
\item
$l$ is the first letter of $\sigma^p (l)$;
\item
$rl$ belongs to $\L (\sigma )$. 
\end{enumerate}
It is then standard to show that there is an admissible fixed point $x\in A^\ZZ$ of $\sigma^p$ verifying $x_{-1}x_0 = rl$.
Moreover, the sequence $x$ is uniformly recurrent \cite{Queffelec:2010}.
 
We will mainly focus on substitutions with specific arithmetic properties. 
We  recall that an algebraic integer $\beta$ is  a {\em Pisot-Vijayaraghan number} if all its algebraic conjugates have a modulus strictly smaller than $1$.

\begin{defi}\label{def:Pisotsub} Let $\sigma$ be a primitive substitution and let $P_{\sigma}$ denote the characteristic polynomial of the incidence matrix $M_{\sigma}$. We say that the  substitution $\sigma$ is 
\begin{itemize} 
\item of {\em Pisot type} (or {\em Pisot} for short) if  $P_{\sigma}$ has a dominant root $\beta >1$ and any  other root $\beta'$ satisfies $0 < \vert \beta' \vert < 1$;
\item of {\em weakly irreducible Pisot type} (or {\em W. I. Pisot} for short) whenever $P_\sigma$ has a real Pisot-Vijayaraghan number as dominant root, its algebraic conjugates,  with possibly $0$  or roots of the unity as other roots;
\item an {\em irreducible} substitution whenever $P_{\sigma}$ is irreducible over $\QQ$;
\item {\em unimodular} if $\det M_{\sigma} = \pm 1$.
\end{itemize}
\end{defi}

For instance the Fibonacci substitution $0 \mapsto 01, 1 \mapsto 0$ and the Tribonacci substitution $1\mapsto 12, 2 \mapsto 13, 3 \mapsto 1$ are unimodular substitutions of Pisot type. 
Whereas the Thue-Morse substitution $0 \mapsto 01$, $1 \mapsto 10$ is a W. I. Pisot substitution.
A less trivial example of a unimodular W. I. Pisot substitution, given by a modified example of B. Solomyak \cite{Solomyak:2021pers}, is the following.
\begin{exemple}\label{ex:Solomyak}\rm
Consider the  substitution 
$$\gamma \colon 0 \mapsto 111120000, 1 \mapsto 11120 \text{ and } 2 \mapsto 0. $$
The eigenvalues of its incidence matrix are  $(1+ \sqrt{2})^{2}, (1- \sqrt{2})^{2}$ and $1$, so $\gamma$ is a unimodular W.I. Pisot substitution. Unlike the Thue-Morse substitution, it is weakly mixing (see Section \ref{sec:discuss} for a proof). 
Thus the unimodular W. I. Pisot assumption  is not sufficient to ensure the discrete spectrum and even the existence of a rotation factor. 
\end{exemple}

Notice that the notions  of Pisot, W. I. Pisot, irreducible, unimodular  depend only on the  properties of the incidence matrix.  
So starting from a Pisot  (resp. W. I. Pisot, irreducible, unimodular) substitution $\sigma$, we get many examples of  Pisot  (resp. W. I. Pisot, irreducible, unimodular) substitutions by permuting the letters in the letter images of $\sigma$.  

Standard algebraic arguments ensure that a Pisot substitution is an irreducible  substitution, and of course, a Pisot substitution is  of weakly irreducible  Pisot type. 

In the following we will strongly use the fact that for any  substitution $\sigma$  of Pisot type (resp. W. I. Pisot, irreducible, unimodular)  and  for every integer $n \ge 1$, the substitutions $\sigma^n$ are also of Pisot type (resp. W. I. Pisot, irreducible, unimodular).

\subsection{Dynamical systems and subshifts}
A {\em measurable dynamical system} is a quadruple $(X, \B, \mu, S)$, where $(X, \B, \mu)$ is a probability space and $S : X \to X$ is a measurable  map  that preserves the measure $\mu$, {i.e.}, $\mu(S^{-1}B) = \mu (B) $ for any $B \in \B$. This system is called {\em ergodic} if any $S$-invariant measurable set has  measure $0$ or $1$.
Two measurable dynamical systems $(X, \B, \mu, S)$ and $(Y, \B', \nu, T) $  are {\em measure theoretically conjugate} if we can find  invariant subsets $X_{0} \subset X$, $Y_{0} \subset Y$ with $\mu(X_{0}) = \nu(Y_{0}) =1$ and a bimeasurable  bijective map $\psi \colon X_{0} \to Y_{0}$ such that $T \circ \psi = \psi \circ S$ and $\mu(\psi^{-1} B) = \nu(B)$ for any $B \in \B'$.

By a {\em topological dynamical system}, or dynamical system for short, we mean a pair $(X,S)$, where $X$ 
is a compact metric space and $S$ a continuous map from 
$X$ to itself.  It is well-known that such a system endowed with the Borel $\sigma$-algebra admits a probability measure $\mu$ preserved by  the map $S$, and then form a measurable dynamical system. 
If $(X,S)$ admits a unique measure preserved by $S$, then the system is said {\em uniquely ergodic}.

The system $(X,S)$ is {\em minimal} whenever $X$ 
and the empty set are the only $S$-invariant closed subsets of $X$. We 
say that a minimal system $(X,S)$ is {\em periodic} whenever $X$ 
is finite.

A dynamical system $(Y,T)$ is called a {\em factor}  of $(X,S)$ if there is a continuous 
and onto  map $\phi: X \rightarrow Y$ such that $\phi \circ S = T \circ \phi$. The 
map $\phi $ is a {\em factor map}.
The system $(X,S)$ is also said to be an {\em extension} of $(Y,T)$.
If $\phi$ is one-to-one we 
say that $\phi $ is a {\em conjugacy}, and, that $(X,S)$ and
$(Y,T)$ are {\em conjugate}.

A {\em Cantor system} is a dynamical system $(X,S)$, where the space  $ X $ is a 
Cantor space, {i.e.},  $ X $ has a countable basis of its 
topology which consists of closed and open sets and does not have 
isolated points.

\medskip

For a finite alphabet $A$, we endow $A^{\ZZ}$ with the product topology.
A {\em subshift} on $A$ is  a 
pair $(X,S_{\mid X})$, where $X $ is a closed $S$-invariant 
subset of $A^{\ZZ}$ ($S(X) = X$) and $S$ is the {\em shift transformation}
\begin{center} 
\begin{tabular}{lllll}
$S$ & : & $A^{\ZZ}$            & $\rightarrow $ & $A^{\ZZ}$ \\
    &   & $(x_n)_{ n \in \ZZ }$ & $\mapsto$      & $(x_{n+1})_{n\in \ZZ }$.
\end{tabular}
\end{center}
We call {\em language} of $X$ the set $\L (X) = \{ x_{[i,j]} ; x\in X, i\leq j\}$.
A set   defined  with two words  $u$ and $v$  of $A^{*}$ by  
$$
[u.v]_X = \{ x\in X ; x_{[-|u|,|v|-1]} = uv \}
$$
is called a {\em cylinder set}. 
When $u$ is the empty word we set $[u.v]_X = [v]_X$.
The family of cylinder sets is a base 
of the induced topology on $X$. 
When no confusion is possible we write $[u]$ and $S$ instead of $[u]_{X}$ and $S_{\mid X}$.

For  $x\in A^\ZZ$, let $\Omega (x)$ be the set 
$\{ y \in A^{\ZZ} ; y_{[i,j]} \in \L(x), \forall \ [i,j]\subset \ZZ \}$. 
It is clear that $(\Omega (x), S)$ is a subshift, it is called  
the {\em subshift  generated by} $x$.  Notice that  $\Omega (x) = \overline{\{ S^n x ; n\in \ZZ \}}$. 
For a subshift 
$(X,S)$  on $A$, the following are equivalent:

\begin{enumerate}
\item
$(X,S)$ is minimal;
\item
For all $x\in X$ we have $X=\Omega (x)$;
\item
For all $x\in X$ we have $\L(X)=\L(x)$.
\end{enumerate}

We also have that $(\Omega (x), S)$ is minimal if and only if 
$x$ is uniformly recurrent.
Note that if $(Y,S)$ is another  subshift then $\L(X) = \L(Y)$ if and only if $X=Y$.

\subsection{Substitution subshifts}

Let  $\sigma$ be a primitive substitution.
We call {\em substitution subshift generated by $\sigma$}  the topological dynamical system $(\Omega (\sigma ), S)$ where $\Omega (\sigma )$ is the set of sequences $x = (x_n)_{n\in \ZZ}$ such that $\L (x) \subset \L (\sigma )$.
Notice that each power $p$ of $\sigma$ generates the same subshift: $\Omega(\sigma)= \Omega(\sigma^{p})$. 

When $\sigma $ is proper we say $(\Omega (\sigma ), S)$ is a {\em proper substitution subshift}.
If the set $\Omega ({\sigma})$ is not finite, the substitution $\sigma$ is called {\em aperiodic}.

In the literature, substitution subshifts are often defined by a different (but equivalent) method, using fixed points. As observed in \cite{Durand&Host&Skau:1999}, the subshift generated by any admissible fixed point  of a power of $\sigma$ is the substitution subshift generated by $\sigma$. Hence $\Omega(\sigma)$ is not empty. Moreover this subshift is minimal and uniquely ergodic (for more details see \cite{Queffelec:2010}).

In  \cite{Holton&Zamboni:1998} is shown that the subshift generated by a unimodular substitution of Pisot type is aperiodic for the shift, and in \cite{Canterini&Siegel:2001}  that the substitution  has to be primitive.
Thus the generated subshift is a  aperiodic minimal Cantor system.

\subsection{Dynamical spectrum of substitution subshifts}

For a measurable dynamical system $(X, \B, \mu, T)$, a  complex
number $\lambda$ is a (dynamical)  {\it eigenvalue} of the dynamical system
$(X,\B, \mu, T)$ with respect to $\mu$ if there exists $f\in L^2(X,\mu)$, $f\not = 0$, such that $f\circ T = \lambda f$; $f$
is called an {\it eigenfunction} (associated with $\lambda$). The value  $1$ is the {\em trivial eigenvalue} associated with a constant eigenfunction. 

If the system is ergodic, then every eigenvalue is of modulus 1, and, every eigenfunction has a constant modulus $\mu$-almost surely.
The collection of eigenvalues is called the {\em spectrum}  of the system, and form a multiplicative subgroup of the circle $\bS =\{ z \in \CC; \ \vert z\vert  =1 \}$. 

For a topological dynamical system, if the eigenfunction $f$ is continuous, the complex number $\lambda$ is called a {\em continuous eigenvalue}. 

An important result for the spectral theory of substitution subshifts is due to B. Host \cite{Host:1986}. It states  that any eigenvalue of a substitution subshift is a continuous eigenvalue. 
 Notice that Host's result was stated for primitive substitutions that are injective on letters but the proof only needed the recognizabilty of the substitution.
From \cite{Mosse:1992} it is known that all primitive substitutions are recognizable whenever its subshift is aperiodic. 
Hence the result stands for all primitive substitutions.

The following proposition, claimed in \cite{Host:1992} (see Proposition 7.3.29 in \cite{Fogg:2002} for a proof), shows that  the spectrum of a unimodular substitution subshift of Pisot type is not trivial.

\begin{prop} \label{prop:vpPisot} 
Let $\sigma$ be  a unimodular substitution of Pisot type and $\alpha$ be a frequency of a letter in any infinite word of $\Omega ({\sigma})$.
Then ${\rm exp}(2i\pi \alpha)$ is a continuous eigenvalue of the dynamical system $(\Omega ({\sigma}), S)$.
\end{prop}
Recall that these frequencies are the coordinates of the right normalized eigenvector associated with the dominant eigenvalue of the incidence matrix of the substitution \cite{Queffelec:2010}, and, moreover, for a unimodular  Pisot substitution they  are multiplicatively independent (Proposition 3.1 in \cite{Canterini&Siegel:2001}). 

\subsection{Domain exchange}\label{subsec:domex}
As we did not find in the litterature a formal definition of domain exchange, we propose below  a definition that fits  for the standard properties satisfied by the interval exchange transformations and the Rauzy fractal. 
Let us recall that a subset of a topological space is said {\em regular} if it equals the closure of its interior.

\begin{defi}\label{def:domainexchange}
We call  {\em domain exchange transformation}  a measurable dynamical system $(E, \B, \lambda_E, T)$,
where  $E$ is a regular compact subset of an Euclidean space, $\lambda_E$ denotes the normalized Lebesgue measure on $E$  and $\mathcal B$ denotes the Borel $\sigma$-algebra, such that: 
\begin{itemize}
\item There exist compact regular subsets $E_{1}$, $\ldots,$ $E_{n}$ such that $ E =E_{1}\cup \cdots \cup E_{n}$;  
\item The sets $E_{i}$ are disjoint in measure:
$$\lambda_E (E_{i} \cap E_{j}) = 0 \hspace{0.5cm} \textrm{ when } i\neq j;$$
\item $\lambda (T(E)) = \lambda (E)$; 
\item   For any $i$, $T$  restricted to $E_{i}$ is a translation. 
\end{itemize}
\end{defi}

Using standard arguments in measure theory (such as translation invariance of the Lebesgue measure), we leave it to the reader to verify  
that for a domain exchange transformation,  the map  $T$ is one-to-one except on a set of measure zero, and the map $T^{-1}$  is measurable.

A compact set $E$ is said {\em self-affine} (by piecewise self-affine maps) when there are finitely many  non surjective piecewise affine maps $f_{1}, \ldots f_{\ell} \colon E \to E$ with a common linear part, such that 
the sets $f_{i}(E)$ are disjoint in measure and 
$ E = \bigcup_{i=1}^{n} f_{i}(E)$.  
We say a domain exchange $(E, \B, \lambda_E, T)$ is {\em self-affine} when $E$ is.

\section{Matrix eigenvalues and return substitutions}\label{sec:returnsub}

In this section we recall the notion of return substitution introduced in  \cite{Durand:1998a} and that any  substitution subshift is conjugate to an explicit primitive and proper substitution subshift without changing too much the eigenvalues of the associated substitution matrix \cite{Durand:1998b}.

\medskip 

Let $A$ be an alphabet, $x$ be an element of $A^\ZZ$ and  $u$ be a word of $x$. 
We call {\em return word} to $u$ of $x$ every factor $x_{[i,j-1]}$, where $i$ and $j$ are two successive occurrences of $u$ in $x$. We denote by $\R_{x,u}$ the set of return words to $u$ of $x$.    
Notice that for a return word $v$, $vu$  belongs to $\L(x)$ and  $u$ is a prefix of the word $vu$. 
Suppose $x$ is uniformly recurrent.
It is easy to check that for any  word $u$ of $x$, the set $\R_{x,u}$ is finite.  
Moreover, for any sequence $y  \in \Omega (x)$, we have $\R_{y,u} =\R_{x,u}$. 
The sequence $x$ can be written naturally as a concatenation

$$ 
x =\cdots  w_{-1}w_{0} w_{1} \cdots, \hspace{1cm} w_{i} \in \R_{x, u}, \  i\in \ZZ,
$$ 
of return words to $u$, and this decomposition is unique. 
By enumerating the elements of $\R_{x, u}$ in the order of their first appearence in $(w_{i})_{i\ge 0}$, we get a bijective map 
$$ 
\Theta_{x, u} \colon R_{x, u} \to \R_{x, u} \subset A^* ,
$$
where $R_{x, u} =\{1, \ldots,  \textrm {Card } (\R_{x, u}  )\}$.
This map defines a morphism from $R_{x, u}^*$ to $A^*$. 
When $u$ is a prefix of $x_{[0,+\infty )}$ we denote by $ D_{u}(x)$ the unique sequence on the alphabet $R_{x,u}$ characterized by 
$$ 
\Theta_{x,u} (D_{u}(x)) = x.
$$

We call it the {\em derived sequence of} $x$ on $u$. 

A finite subset $\R \subset A^+$ is a {\em code} if  every word $u \in A^+$ admits at most one decomposition as a concatenation  of elements of $\R$. 

We say that a code $\R$ is a {\em circular code} if for any words
$$  
w_{1}, \ldots, w_{j}, w, w'_{1}, \ldots, w'_{k} \in \R; s\in A^+ \textrm{  and } t \in A^*
$$
such that
$$
w=ts \textrm{ and } w_{1}\cdots w_{j} = sw'_{1} \cdots w'_{k}t
$$
then $t$ is the empty word. It follows that $j=k+1$, $w_{i+1} =w'_{i'}$ for $1 \le i \le k$ and $w_{1}=s$. 

\begin{prop}[\cite{Durand&Host&Skau:1999} Lemma 17]
\label{prop:circular}
Let $x\in A^\ZZ$ be a uniformly recurrent sequence and $u$  a prefix of $x_{[0, +\infty)}$. 
The set $\R_{x,u}$ is a circular code.
\end{prop}

The next four propositions are usually stated for one-sided sequences,
but they are still true for uniformly recurrent $x\in A^\ZZ$ as these statements only depend on $x^+ = x_{[0,+\infty )}$.
Indeed,  $x^+$ is also a uniformly recurrent element of $A^\NN$.
We can define in the same way the derived sequence $D_u (x^+)$, the sets $\R_{x^+,u}$ and $R_{x^+,u}$, and, the map $\Theta_{x^+,u}$ when $u$ is a prefix of $x^+$.
Moreover we clearly have: 
$$
D_u (x^+) = D_u (x)_{[0,+\infty )},  \R_{x^+,u} = \R_{x,u},  R_{x^+,u}= R_{x,u} \hbox{ and } \Theta_{x^+,u} = \Theta_{x,u} .
$$

The following proposition  enables to associate to a substitution  another substitution on the alphabet $R_{x,u}$. 
We recall that admissible fixed points of primitive substitutions are uniformly recurrent. 
So the return word notion are meaningful.

\begin{prop}[\cite{Durand:1998a} proof of Proposition 5.1]
\label{subst_retour} 
Let $x\in A^\ZZ$ be an admissible fixed point of the primitive substitution  $\sigma$ and $u$ be a nonempty prefix of $x_{[0,+\infty )}$. 
There exists a primitive substitution $\sigma_{u}$, defined on the alphabet $R_{x,u}$, characterized by
$$
\Theta_{x,u} \circ  \sigma_{u} = \sigma \circ  \Theta_{x,u}.
$$
\end{prop}     
For a prefix $u$ of $x_{[0,+\infty )}$, where $x \in A^\ZZ$ is an admissible fixed point of a primitive substitution $\sigma $, the derived sequence $D_{u}(x)$ is an admissible fixed point of the substitution $\sigma_{u}$. 
Indeed, we have
\begin{align*}
\Theta_{x,u}  \circ \sigma_{u} (D_{u}(x)) = \sigma \circ \Theta_{x,u}( D_{u}(x) ) = \sigma (x) = x = \Theta_{x,u}  \circ  D_{u}(x).
\end{align*}
Then, the unicity of the concatenation into return words implies  $D_{u}(x)$ is a fixed point of $\sigma_{u}$.
It is clearly admissible as it is uniformly recurrent.

This substitution, defined in the previous proposition, is called the {\em return substitution} ({\em to } $u$). 
Moreover, we observe that for any integer $l >0$
$$
(\sigma^l)_{u}  = (\sigma_{u})^l.
$$
 
Furthermore the  incidence matrix of the return substitution has almost the same spectrum as the initial substitution. More precisely, we have:
\begin{prop}[\cite{Durand:1998b} Proposition 9]
\label{subst_retourvp}
Let $\sigma$ be a primitive substitution and let $u$ be a prefix of $x_{[0,+\infty )}$ where $x\in A^\ZZ$ is an admissible fixed point of $\sigma $. 
The incidence  matrices $M_{\sigma }$ and $M_{\sigma_{u}}$ have the same eigenvalues, except perhaps zero and  roots of the unity. 
\end{prop}
For instance, for the Tribonacci substitution $\tau$, the return substitution $\tau_{1}$ is the same as $\tau$. On the other hand, if we consider the substitution 
$$\sigma \colon 1 \mapsto  1123,  2\mapsto  211, \textrm{ and } 3 \mapsto 21,$$ it is also a substitution of Pisot type and the incidence matrix of the return substitution $\sigma_{11}$ has $0$ as eigenvalue.   
Indeed, the characteristic polynomial of $\sigma $ is $X^3 - 3X^2 - X - 1$ and one has 
$$\Theta_{x,11} \colon 1 \mapsto  1123,  2\mapsto  11232, 3 \mapsto 122, 4 \mapsto 1 \textrm{ and } 5 \mapsto 11212,$$
$$\sigma_{11} \colon 1 \mapsto  1234,  2\mapsto  12544, 3 \mapsto 1244, 4 \mapsto 1  \textrm{ and } 5 \mapsto 1244244,$$
with the characteristic polynomial of $\sigma_{11}$ being $X(X + 1) (X^3 - 3X^2 - X - 1)$.

The next proposition is a key statement in the proof of our main result Theorem \ref{theo:main}.
It allows to work with a proper substitution.

\begin{prop}
\label{expansion} 
Let $y = (y_{i})_{i\in \ZZ}$ be an admissible fixed point of a  primitive substitution $\tau$ on the alphabet $R$. Let $\Theta : R^* \to A^*$ be a non-erasing morphism, $x= \Theta(y)$ and $(X,S)$ be the subshift generated   by $x$.

Then, there exist a primitive substitution $\xi$ on an alphabet $B$, an admissible fixed point $z$ of $\xi$, and a map $ \phi : B \to A$  such that:
\begin{enumerate}
\item 
\label{item:substitutive}
$\phi(z) = x$;
\item 
\label{item:circularcode}
If $\Theta(R)$ is a circular code, then $\phi$ is a conjugacy from $(\Omega ({\xi}) , S)$ to $(X,S)$;
\item 
\label{item:proper}
If $\tau$ is proper (resp. right or left proper), then $\xi$ is proper (resp. right or left proper);
\item
\label{item:equal power}
There exists a prefix $u \in B^+$  of $z_{[0,+\infty )}$ such that $R_{y,y_{0}} = R_{z,u}$ and there is an integer $l \ge 1$  such that  the return substitutions $\tau_{y_{0}}^l$ and $\xi_{u}$ are the same. 
\end{enumerate} %
\end{prop} 
Actually  the first three statements of this proposition  correspond to Proposition 23 in \cite{Durand&Host&Skau:1999}. The substitution $\xi$ is explicit in the proof. 

\begin{proof} The statements \eqref{item:substitutive}, \eqref{item:circularcode}, \eqref{item:proper}, and the fact that $\xi$ is primitive, have been proven in \cite[Proposition 23]{Durand&Host&Skau:1999}.
We will just give the proof of the first statement because we need it to prove the fourth statement. 

Considering a power of $\tau$ instead of $\tau $ if needed, we can assume that   $\vert \tau(j) \vert \ge  \vert \Theta (j) \vert   $ for any $j \in R$. 
For all $j \in R$, let us denote $m_{j} = \vert \tau(j) \vert$ and $n_{j} = \vert \Theta(j) \vert$. 
We define 
\begin{itemize}
\item An alphabet $B := \{ (j, p) ;  j\in R, 1 \le p \le n_{j} \}$; 
\item A morphism $\phi \colon B^* \to A^*$ by $\phi(j,p) = (\Theta( j))_{p}$;
\item A morphism $\psi \colon R^* \to B^*$ by $\psi(j) =(j,1)(j,2) \cdots (j, n_{j})$.
\end{itemize}

Clearly, we have  $\phi \circ \psi = \Theta$. We define a substitution $\xi$ on $B$ by
\begin{eqnarray*}
\forall j \in R, \ 1 \le p \le n_{j} ; \ \xi(j,p )= \begin{cases} \psi((\tau(j))_{p})  & \textrm{ if } 1 \le p < n_{j}\\ \psi((\tau(j))_{[n_{j},m_{j}]}) & \textrm{ if } p= n_{j}. \end{cases} 
\end{eqnarray*}

Thus  for every $j \in R$, we have  $\xi(\psi(j)) = \xi(j,1)\cdots \xi(j, n_{j}) = \psi(\tau(j))$, {i.e.},

\begin{equation}
\label{eq:returnsub}
\xi \circ \psi = \psi \circ \tau.
\end{equation}

For $z = \psi(y)$ we obtain $\xi(z) = \psi(\tau(y)) = \psi(y) = z$, that is  $z$ is a fixed point of $\xi$.  
Moreover, $\phi(z) = \phi(\psi(y)) = \Theta(y) =x$ and we get the point (1). 
   
Let us prove the fourth statement.

Let $u =\psi(y_{0}) \in B^*$.  First, notice the morphism  $\psi$ is one-to-one and then we have $\psi(\R_{y, y_{0}}) = \R_{\psi(y), \psi(y_{0})}$. It follows that 
$$ R_{y,y_{0}} = R_{\psi(y), \psi(y_{0})} = R_{z,u},$$
and 
$$
\psi \circ \Theta_{y,y_{0}} = \Theta_{\psi(y), \psi(y_{0})} = \Theta_{z,u}.
$$

Therefore for the  return substitution  $\tau_{y_{0}}$ to $y_{0}$,  Proposition \ref{subst_retour} and Relation (\ref{eq:returnsub})  give 
$$
\Theta_{z, u} \circ  \tau_{y_{0}} = \psi  \circ  \Theta_{y,y_{0}}  \circ \tau_{y_{0}}
= \psi  \circ  \tau \circ  \Theta_{y,y_{0}} 
= \xi \circ  \psi \circ  \Theta_{y,y_{0}}
=\xi \circ  \Theta_{z, u}.
$$

Consequently, we have $ \tau_{y_{0}} = \xi_{u}$.
\end{proof}

As a corollary of propositions \ref{prop:circular}, \ref{subst_retour}, \ref{subst_retourvp} and \ref{expansion}, we get
 
\begin{coro}
\label{coro:wPisot}  
Let $\sigma $ be a  primitive substitution. 
Then there exists a proper primitive substitution $\xi$ on an alphabet $B$, such that 

\begin{enumerate} 
\item 
$(\Omega ({\sigma}), S)$ is conjugate to $(\Omega ({\xi}), S)$;
\item 
there exists $l\geq 1$ such that the substitution matrices $M_{\sigma}^l$ and $M_{\xi}$ have the same eigenvalues, except perhaps 0 and 1.
\end{enumerate}
\end{coro}

\begin{proof}  Considering a power of $\sigma$ instead of $\sigma$ if needed, we can assume it 
has an admissible fixed point $x$. Let us fix a nonempty prefix $u$ of $x_{[0,+\infty )}$.
 We  can also  assume  that  the word $\Theta_{x,u}(1)u$ is a prefix of  $\sigma(u)$.   
By the very definition of return word, for any letter $i \in R_{x,u}$, the word $\Theta_{x,u}(i)u$ has the word $u$ as a prefix.   
Then $\Theta_{x,u}(1)u$ is a  prefix of the word $\sigma (\Theta_{x,u} (i) u)$. 
It follows from  the equality in Proposition \ref{subst_retour},  that  
$\Theta_{x,u}(1)u$ is also a prefix  of the word $\Theta_{x,u}\circ  \sigma_{u}(i)$. The uniqueness of the coding by $\Theta_{x,u} (R_{x,u})$, implies that the word $\sigma_{u}(i)$ starts with $1$, and the substitution $\sigma_{u}$ is left proper. 

Proposition \ref{expansion} (together with Proposition \ref{prop:circular}), applied to the admissible fixed point $y=D_u (x) $ of $\sigma_u$ and to $\Theta = \Theta_{x,u}$, gives the existence
of a left proper primitive substitution $\xi': C^* \to C^*$ such that $(\Omega ({\sigma}), S)$ is conjugate to $(\Omega ({\xi'}), S)$.
Moreover, by Proposition  \ref{subst_retourvp}  and \ref{expansion} there exists an integer $l>0$ such that the incidence matrices $M_{\sigma}^l$ and $M_{\xi'}$ share the same eigenvalues, except perhaps 0 and 1.

Let us recall some morphism relations we have from the definitions and the proof of Proposition \ref{expansion} taking $\sigma_u$ for $\tau$, $\Theta_u$ for $\Theta$:
\begin{align*}
\Theta_u \circ \sigma_u & =  \sigma \circ \Theta_u, \  \xi' \circ \psi =  \psi \circ \sigma_u , \  \phi \circ \psi  = \Theta_u 
\end{align*}
and thus 
\begin{align}
\label{eq:boulette}
\phi \circ \xi' \circ \psi  & = \sigma \circ \Theta_u .
\end{align}

To obtain a proper substitution we need to modify $\xi' $.
Let $a\in C$ be the letter such that for all letter $b\in C$, $\xi' (b)=  aw(b)$  for some word $w(b)$. 
Now consider the substitution $\xi'' : C^* \to C^*$ defined by $\xi'' \colon b \mapsto w(b)a$.
Observe that $\xi''$ is primitive and, for all $n$, one gets $\xi'^n (b) a =  a\xi''^n (b) $.
Then, $\xi'$ and $\xi''$ define the same language,  so we have  $\Omega ({\xi'}) = \Omega ({\xi''})= \Omega ({\xi})$, where $\xi = \xi' \circ \xi''$.
Thus $\xi $ is clearly proper. 
We conclude observing that $M_{\xi} = M_{\xi'} M_{\xi''} = M_{\xi'}^2$.
 
\end{proof}

In terms of Pisot substitutions, Corollary \ref{coro:wPisot} becomes: 

\begin{coro}\label{cor:Pisotsub}
Let $\sigma$ be a substitution of Pisot type, then the substitution subshift associated with $\sigma$ is conjugate to  a substitution subshift $(\Omega ({\xi}), S)$, where $\xi$ is a proper primitive substitution of weakly irreducible Pisot type.

Moreover, the spectrum of its incidence matrix $M_{\xi}$ is that of a power of $M_{\sigma}$ except perhaps 0 and 1.   
\end{coro}

The example after Proposition \ref{subst_retourvp} shows that the use of return substitutions seems to force to deal with W. I. Pisot substitutions. 
In fact, it is unavoidable to consider W. I. Pisot substitution to represent a Pisot substitution subshift by a proper substitution. 
For instance, consider the  non-proper substitution $\sigma : 0 \mapsto 001,$ $1 \mapsto 10$. The dimension group of the associated subshift, computed in \cite{Durand:1996},  is of rank $3$. As a consequence, any proper substitution $\xi$ representing  the subshift $\Omega ({\sigma})$ should be, at least, on $3$ letters (see \cite{Durand&Host&Skau:1999} for the details). 
Moreover, Cobham's theorem (see Theorem 14 in \cite{Durand:1998c}) for minimal substitution subshifts implies that, taking powers if needed, $\xi$ and $\sigma$ share the same dominant eigenvalue. So, the substitution $\xi$ can not be irreducible.

\section{Conjugacy with a domain exchange}\label{sec:domex}
In this section we give sufficient conditions  on a  primitive proper substitution so that the associated substitution subshift is measurably conjugate to a domain exchange in an Euclidean space. 
\subsection{Using Kakutani-Rohlin partitions}\label{subsec:KR}

In this subsection, we will assume that $\xi$ is a primitive proper  aperiodic substitution  on a finite  alphabet $A_{\xi}$ equipped with a fixed order.

First let us recall a structural property of the system $(\Omega ({\xi}), S)$ in terms of Kakutani-Rohlin partitions.

\begin{prop}[\cite{Durand&Host&Skau:1999}]\label{substKR}
Let $\xi$ be an aperiodic primitive proper substitution on a finite alphabet $A_{\xi}$. Then for every $n > 0$,  
$$ 
\P_{n} =\{S^{-k} \xi^{n-1} ([a]) ; \  a \in A_{\xi}, \ 0 \le k \le \vert \xi^{n-1} (a) \vert -1 \}
$$
 is a clopen partition of $\Omega ({\xi})$ defining a nested sequence of Kakutani-Rohlin partitions of $\Omega ({\xi})$, that is to say:
\begin{itemize}
\item 
The sequence $(\xi^n(\Omega ({\xi})))_{n\ge 0}$ is decreasing and the intersection is only one point;
\item 
For every $n>0$, $\P_{n+1}$ is finer than $\P_{n}$;
\item 
The sequence $(\P_{n})_{n>0}$ spans the topology of $\Omega ({\xi})$.
\end{itemize}
\end{prop}
The fact that the family of clopen sets $\P_{n}$ is a partition is an immediate consequence Moss\'e's theorem asserting that aperiodic primitive substitutions are recognizable (see \cite{Mosse:1992,Mosse:1996}).

 To be coherent with the notations in \cite{Bressaud&Durand&Maass:2005}, we take the conventions $\P_{0} = \{ \Omega ({\xi}) \} $ and 
for an integer $n \ge 1$, $r_{n}(x) $ denotes the {\em entrance time} of a point $x \in \Omega ({\xi})$ in 
$\xi^{n-1}(\Omega ({\xi}))$, that is
$$
r_n (x) = \min \{ k\geq 0 ; \ S^k x \in \xi^{n-1} (\Omega (\xi ) ) \}. 
$$

By minimality, this value is finite for any $x \in \Omega ({\xi})$ and the function $r_{n}$ is continuous.

The homeomorphism $ S_{\xi(\Omega ({\xi}))} \colon \xi(\Omega ({\xi})) \ni x \mapsto S^{r_{2}(Sx)}(Sx) \in \xi(\Omega ({\xi}))$ is what is usually called the induced map of the system $(X,S)$ on the clopen set $\xi(\Omega ({\xi}))$. Since we have the relation 
\begin{align}\label{eq:sysinduit}
\xi \circ S =  S_{\xi(\Omega ({\xi}))} \circ \xi,
\end{align}
the induced system $(\xi(\Omega ({\xi})), S_{\xi(\Omega ({\xi}))})$ is a factor of $(\Omega ({\xi}),S)$ via the map $\xi$ (and in fact a conjugacy  when $\xi$ is aperiodic).

Note that for any integer $n>0$,
\begin{align}\label{eq:tpsretour}
r_{n}(Sx) - r_{n}(x) = \begin{cases} -1 & \textrm{if } x \not\in \xi^{n-1}(\Omega ({\xi})) \\ 
 \vert \xi^{n-1}(a) \vert -1 & \textrm{if } x \in \xi^{n-1}([a]), a\in A_{\xi}. 
   \end{cases}  
\end{align}

More precisely, we can relate the entrance time and the incidence matrix by the  following equality (see Lemma 4 in \cite{Bressaud&Durand&Maass:2005}): For a primitive proper substitution $\xi$, we have  for any $x \in \Omega ({\xi})$ and $n \ge 2$
\begin{align}\label{eq:rn}
r_n(x) = \sum_{k=1}^{n-1}
\langle s_k (x) ,(M_\xi^t)^{k-1} H(1)\rangle ,
\end{align}

where $\langle \cdot, \cdot \rangle$ denotes the usual scalar product, $M_{\xi}^t$ is the transpose of the incidence matrix,  $H(1) =(1, \ldots, 1)^t $ and  $s_k : \Omega (\xi ) \to \mathbb{Z}^{A_{\xi}}$ is a continuous function defined by
$$
s_k (x)_a = \# \{ r_k (x) < i \leq r_{k+1} (x) ; \  S^i x \in \xi^{k-1} ([a])  \}, \hspace{1cm} \textrm{for }a \in A_{\xi}.
$$

The sequence of vectors $(s_{k}(x))_{k}$ provides a coding of the orbit of points that is analog to the prefix-suffix expansion \cite{Fogg:2002}.
In other words, the vector $s_{k}(x)$ counts, in each coordinate $a \in A_{\xi}$, the number of time that the positive iterates of $x$ meet the clopen set $\xi^{k-1}([a])$ until meeting, for the first time, the clopen set $\xi^{k}(\Omega ({\xi}))$ and after meeting the clopen set $\xi^{k-1}(\Omega ({\xi}))$. 

The proof of the following lemma is direct from the  definitions of $s_k$ and $r_k$ and Proposition \ref{substKR}.

\begin{lemma} 
\label{lemma:lemmeclassic}
For $\xi$ a primitive proper aperiodic substitution, we have, for any $x \in \Omega ({\xi})$,
\begin{align*}
s_{1}(\xi (x)) = 0 \hspace{0.5cm} \textrm{ and } \hspace{0.5cm} \forall k >1, \ 
s_{k}(\xi (x)) = s_{k-1}(x). 
\end{align*}
For every letter $a \in A_{\xi}$, $k\in \NN^*$,  we also  have  $ s_{k} (x)_{a} \le \sup_{b \in A_{\xi}} \vert \xi(b) \vert$. 
\end{lemma}

From the ergodic point of view,  it is well-known (see \cite{Queffelec:2010}) that subshifts generated by primitive substitutions are uniquely ergodic.
We call $\mu$ the unique  probability  shift-invariant measure of $(\Omega (\xi ) , S)$. We have the following relations, for any positive integer $n$,
\begin{align}\label{eq:measure}
\vec{\mu}({n}) = M_{\xi} \vec{\mu}({n+1}), \hspace{0.5cm} \textrm{ and } \hspace{0.5cm} \langle H(1), \vec{\mu}(1) \rangle = 1,
\end{align}
where $\vec{\mu}({n}) \in \RR^{A_{\xi}}$ is the vector defined by
$$ \vec{\mu}(n)_{a} = \mu(\xi^{n-1}([a])), \hspace{1cm} \textrm{ for every letter } a \in A_{\xi}.$$  
It is well-known \cite{Queffelec:2010} that $\vec{\mu} (1)$ is a Perron eigenvector of the dominant 
eigenvalue $\theta$ of $M_\xi$ and $\vec{\mu} (n)=\theta^{n-1} \vec{\mu} (1)$.



\subsection{On the spectrum  of a substitution subshift}\label{subsec:spectrum}

From this subsection, we assume that $\xi $ is a  primitive proper substitution  on a finite  alphabet $A_{\xi}$ and $\mu$ be the unique invariant probability measure of $(\Omega (\xi),S)$.

Fix some $ \rho_1<1$ greater than the maximum of the modulus of eigenvalues  of $M_\xi$ smaller than $1$. 
Taking a power of $\xi$ if needed,  from classical results of linear algebra, there are $M_{\xi }^{t}$-invariant $\RR$-vectorial subspaces $E^0, E^{wu}$ and $E^s$, a norm $\| \cdot\|_{\xi}$  and constants  $C_{s}, C_{wu},  >0$   such that  for any $n \in \NN$ 

\begin{enumerate}
\item
$\RR^{A_{\xi}} = E^0 \oplus E^s \oplus E^{wu}$, 
\item
$\ker M_\xi^t = E^{0}$,

\item  
$ 0<\|(M_\xi^{t})^n v \|_{\xi}  \le C_{s}  \rho_{1}^n  \|v\|_{\xi}$,  for all $v\in E^s\setminus\{0\}$,
\item  
$\| (M_\xi^{t})^n v\|_{\xi }\ge C_{wu}  ((1+\rho_{1})/2)^{n}  \|v\|_{\xi}$,  for all $v\in E^{wu}$.\end{enumerate}

Roughly speaking, the (stable) space $E^{s}$ corresponds to  the eigenspaces associated to non zero eigenvalues with modulus strictly less than $1$.
The (weakly unstable) space $E^{wu}$, corresponds to  the eigenspaces associated to non zero eigenvalues with modulus greater or equal to $1$.

Let us apply some well-known arguments (see \cite{Host:1986, Ferenczi&Mauduit&Nogueira:1996} for substitutions and \cite{Bressaud&Durand&Maass:2005} for a wider context).
Let $r_n$ and $s_n$ be as defined in Section \ref{subsec:KR}. 

\begin{prop}\label{prop:condcont}  Let $\xi$ be a primitive proper substitution on an alphabet $A_{\xi}$. 
Let  $\lambda \in \bS$.  The following statement are equivalent.
\begin{itemize}
	\item The complex number $\lambda $ is an eigenvalue of  the system $(\Omega (\xi ) , S)$.
	\item The sequence $(\lambda^{-r_n}) _{ n\geq 1}$ converges uniformly to a continuous  eigenfunction associated with $\lambda$.
 	\item The sum $\sum_{n\ge 1} \max_{a\in A_{\xi}} |\lambda^{|\xi^n (a)|}-1|$ converges.
	\item For any $a\in A_{\xi}$, the sequence $(\lambda^{|\xi^n (a)|})_{n \ge 0}$ converges to $1$ when $n$ goes to infinity.
\end{itemize}
\end{prop}

So if $\exp(2i\pi\alpha)$ is an eigenvalue of the substitution subshift $(\Omega ({\xi}), S)$,  for any letter $a$ of the alphabet $\vert \xi^{n} (a) \vert \alpha$ converges to $0$ mod $\ZZ$  as $n$ goes to infinity. In an equivalent way the vector $(M_{\xi}^t)^n \alpha (1,\ldots, 1)^t $ tends to $0$ mod $\ZZ^{A_{\xi}}$. The next lemma precises this  for the usual convergence.    

\begin{lemma}
\label{lemme:condcont}  
Let $\xi$ be a primitive proper substitution on an alphabet $A_{\xi}$ and let $\lambda=\exp(2i\pi\alpha)$ be an eigenvalue of the substitution subshift  $(\Omega (\xi ) ,S)$. 
Then, there exist $m \in \NN$, $v, w \in
\RR^{A_{\xi}}$  such that
\begin{align}\label{eq:lem14}
\displaystyle \alpha H(1)=v + w, \hspace{1cm} (M_{\xi}^{t})^{m} w \in \ZZ^{A_{\xi}} \hbox{ and } (M_\xi^t)^{n} v
\rightarrow_{n\to \infty} 0,
\end{align}

where all entries of $H(1)$ are equal to $1$. Moreover, for every  vectors $v,w$ satisfying \eqref{eq:lem14}
\begin{enumerate}[label= \roman*)] 
\item\label{it:L14i} The convergence is geometric: there exist $0 < \rho \le \rho_{1} < 1$ and a constant $C$ such that 
$$ \vert\vert   (M_\xi^t)^{n} v \vert \vert \le C \rho^n, \textrm{ for any } n\in \NN.$$
\item\label{it:L14ii} For any positive integer $n$, 
$$  \langle v, \vec{\mu}(n) \rangle =0  \hspace{0.5cm} \textrm{ and } \hspace{0.5cm}  \alpha = \langle (M^{t}_{\xi})^{n-1} w, \vec{\mu}(n) \rangle.$$
\end{enumerate}
\end{lemma}
 
 \begin{proof}
 The first claim \eqref{eq:lem14} comes from \cite[Lemme 1]{Host:1986} and  Item \ref{it:L14i} follows from the observation we made before the statement of the lemma. 
We just need to show  Item \ref{it:L14ii}.
 Notice that  the relations (\ref{eq:measure}) give us for any positive integer 
 \begin{align*}
 \langle v ,\vec{\mu}(n) \rangle = \langle v, M_{\xi}^{p} \vec{\mu}(n+p) \rangle = \langle  (M_{\xi}^{t})^{p} v, \vec{\mu}(n+p) \rangle \to_{ p\to +\infty} 0.
 \end{align*}
 Then, we deduce
 \begin{align*}
  \alpha & =  \alpha \langle H(1), \vec{\mu}(1) \rangle = \langle v, \vec{\mu}(1) \rangle + \langle w, \vec{\mu}(1) \rangle \\
  & = \langle w, \vec{\mu}(1) \rangle =    \langle (M^{t}_{\xi})^{n-1} w, \vec{\mu}(n) \rangle.
  \end{align*}
 
 \end{proof}

Observe that the decomposition $\alpha H(1)= v+w$ is not unique since for every $u \in E^0$, we have $\alpha H(1) = (v-u) + (w+u)$ is another decomposition fulfilling the properties \eqref{eq:lem14}.

\medskip

If $\exp(2i\pi \alpha_{1} ), \ldots, \exp(2i \pi \alpha_{d})$ are $d$ eigenvalues of the substitution subshift $(\Omega ({\xi} ), S)$, from Proposition \ref{prop:condcont} and Lemma \ref{lemme:condcont} there exist $m \in \NN$, $v (1), \dots v(d),$ $w (1) , \dots ,w (d) \in
\RR^{{A_{\xi}}}$  such that for all $i \in \{ 1, \dots , d\}$:

\begin{equation}\label{eq:wetv}
 \alpha_i  H(1)=v(i) + w(i), \hspace{0.5 cm} (M_\xi^{t})^m w(i) \in \ZZ^{A_{\xi}} \hbox{ and } \sum_{n\ge 1} (M_\xi^{t})^n v(i)
\textrm{ converges}. 
\end{equation}

Notice that up to take a power of $\xi$,  we can assume below  that  the constant $m$ is equal to $1$. By the observation made after Lemma \ref{lemme:condcont}, we also assume that each  $v(i)$ has no component in $E^{0}$. 

\medskip

We recall Proposition \ref{prop:vpPisot}: a unimodular Pisot substitution subshift on the alphabet $A$ admits $\#A-1$ non trivial  eigenvalues
$\exp(2i\pi \alpha_{1} )$, $\ldots,\exp(2i \pi \alpha_{\#A-1})$  that are {\em multiplicatively independent}, {i.e.} the values $1, \alpha_{1}, \ldots, \alpha_{\#A-1}$ are rationally independent.
This motivates the next proposition that interprets the arithmetical properties of the eigenvalues  in terms of the vectors $v(i)$ and $w(i)$. 
\begin{prop}
\label{prop:linind}  Let $\xi$ be a primitive proper substitution. If the complex numbers $\exp(2i\pi \alpha_{1} )$, $ \ldots,$ $\exp(2i \pi \alpha_{d})$ are $d$ multiplicatively independent eigenvalues  of the substitution subshift $(\Omega ({\xi} ), S)$.
Then, 
both families of vectors 
$$\{M_{\xi}^tv(1),\ldots, M_{\xi}^tv ({d}) \} \hbox{ and } \{M_{\xi}^tH(1),M_{\xi}^t w({1}), \ldots, M_{\xi}^t w({d}) \}$$  
are linearly independent. 
\end{prop}
Notice it also implies that both family of vectors $\{v(1),$ $\ldots$, $v ({d}) \}$ and $\{H(1),$ $w({1})$, $\ldots$, $w({d}) \}$  are linearly independent.

\begin{proof}
This proof is similar to the proof of Proposition 10 in \cite{Bressaud&Durand&Maass:2010}. 
We adapt it to our context, since it does not straightforwardly apply.

Assume there exist real numbers $\delta_{0},  \delta_{1}, \ldots, \delta_{d} $,  one being different from $0$,  such that $\delta_{0} M^{t}_{\xi} H(1)+ \sum_{i=1}^{d} \delta_{i }M_{\xi} ^{t}w({i}) = 0$. 
Since all the vectors are in $\ZZ^{A_{\xi}}$,  we can assume that every $\delta_{i}$ is an integer. 
Taking the inner product of this sum with the vector $\vec{\mu(2)}$, the normalization  and recurrence relations of this vector (Relation (\ref{eq:measure}))  together with  the normalization with respect to  each $w({i})$ in item $ii)$ of Lemma \ref{lemme:condcont},  give us 
$\delta_{0}+ \sum_{i=1}^{d} \delta_{i} \alpha_{i} =0$. 
The rational independence  of the numbers $1, \alpha_{1}, \ldots, \alpha_{d}$ implies that each $\delta_{i}$ equals $0$. 
So the vectors $M_{\xi}^{t}H(1), M_{\xi}^{t}w({1})$, $\ldots$, $M_{\xi}^{t}w({d})$ are independent.

Now, let's assume  there are real numbers $\lambda_{i}$ such that $\sum_{i=1}^{d}\lambda_{i} M_{\xi}^t v({i}) =0$. 
We obtain  $(\sum_{i=1}^{d}\lambda_{i} \alpha_{i})M_{\xi}^t H(1) - \sum_{i=1}^{d} \lambda_{i} M_{\xi}^t w({i}) =0$. 
The independence of the vectors   $M_{\xi}^t H(1), M_{\xi}^t w({1}), \ldots$, $ M_{\xi}^t w({d})$ implies that $\lambda_{i} =0$ for any $i$. So the vectors $M_{\xi}^t v({1}), \ldots, M_{\xi}^t v({d})$ are independent.
\end{proof}

The following property gives a bound on the number of multiplicatively independent eigenvalues for a substitution subshift.
We denote by ${\rm Vect}( u(1) , \dots , u(n))$ the vectorial subspace  spanned by vectors $u(1), \dots , u(n)$.

\begin{prop} \label{prop:ME} Let $\xi$ be a proper primitive substitution. 
If the substitution subshift $(\Omega ({\xi}), S)$ admits $d$ eigenvalues  $\exp(2i\pi \alpha_{i})$, $1\leq i\leq d$, then  the vectorial space spanned by the vectors $v(i)$, 
$$E_{\xi} = {\rm Vect}( v (1), \ldots,  v (d)),$$  is a subspace  of $E^{s}$.
Moreover, if   the eigenvalues are multiplicatively independent,   then  $d\le \dim  E^{s}$. 
\end{prop}

\begin{proof}
For $i\in \{1, \dots , d \}$,
the vector $v(i)$ can be decomposed using the $\RR$-vectorial subspaces $E^0, E^{wu}$ and $E^s$ (see the beginning of Section \ref{subsec:spectrum}).
From Lemma \ref{lemme:condcont} and since the norms $\| \cdot \|$ and $\| \cdot \|_{\xi}$ are equivalent, it has no component in $E^{wu}$.
From the choice we made in \eqref{eq:wetv}, it has no component in $E^0$.
Thus $v(i)$ belongs to $E^s$. So we get $E_{\xi}  \subset E^s$. The  bound is obtained  with Proposition \ref{prop:linind}.
\end{proof}

To construct the domain exchange of a unimodular Pisot substitution subshift we will need the following direct corollary.

\begin{coro} \label{coro:ME}  Let $\xi$ be a proper primitive substitution. 
If the substitution subshift $(\Omega ({\xi}), S)$ admits $\dim  E^{s}$ multiplicatively independent eigenvalues, then 
$ E_{\xi} := {\rm Vect}( v (1), \ldots, v (\dim  E^{s}))  = E^{s}$. In particular, we have that  $M_{\xi}^{t} (E_{\xi})$ equals $E_{\xi}$.
\end{coro}
Notice that for a  unimodular Pisot substitution $\sigma$,  $\textrm{dim } E^{s} +1$ equals the degree of the associated Pisot number, or the number of letters in the alphabet, and the eigenvalues of $M_{\sigma}$ are all simple \cite{Canterini&Siegel:2001}. Thus, by Proposition  \ref{prop:vpPisot},  the proper W. I. Pisot substitution $\xi$ associated   to $\sigma$ in Corollary \ref{cor:Pisotsub}, fulfills the conditions of  Corollary \ref{coro:ME}.

\subsection{Discussion on working hypothesis}\label{sec:discuss}
In the next section we prove Theorem \ref{theo:main} as a corollary of a more general result obtained under the following hypothesis. 

\medskip

{\bf Hypotheses \texttt{P}.} {\it  Let $\xi$ be a  primitive  substitution on a finite alphabet $A_{\xi}$ such that:
\begin{enumerate}[label={\texttt{P}.\roman{enumi}}),ref=\texttt{P}.\roman{enumi}]

\item
\label{hypoii}  
The substitution subshift $(\Omega ({\xi}), S)$  admits  $d_{\xi} =\dim  E^{s} $ eigenvalues $\exp(2i\pi \alpha_{1} )$,  $\ldots,$ $\exp(2i \pi \alpha_{d_{\xi}})$ such that $1$,   $\alpha_{1}, \ldots, \alpha_{d_{\xi}}$  are rationally independent. 
\item\label{hypoiii} 
Its  Perron number $\beta$ satisfies $\beta \vert \det M^{t}_{\xi\vert E^{s}} \vert =1$. 
\end{enumerate}
}

Before to state this result let us discuss these hypotheses. 

All these hypotheses apply to the proper substitution $\xi$ of  Corollary  \ref{cor:Pisotsub} associated with a unimodular Pisot substitution on the alphabet $A$: The existence of  $\#A-1$ multiplicatively independent  (dynamical) eigenvalues comes from Proposition \ref{prop:vpPisot} and the observation made after it. Moreover, Corollary \ref{coro:wPisot} ensures $\dim E^{s} = \#A-1$, and then provides the property \eqref{hypoii}.    
The property  \eqref{hypoiii} is due to the fact that the space $E^{s}$ is spanned by the eigenspaces associated with the algebraic conjugates $\beta_{1}, \ldots, \beta_{\#A-1}$  of the leading  eigenvalue $\beta$ of $M_{\xi}$ that is a Pisot number. 
The unimodular hypothesis implies $\vert \beta \beta_{1} \cdots \beta_{\#A-1} \vert =1$.

It is interesting to recall that the  W.I. Pisot hypothesis  is not sufficient to ensure \eqref{hypoii} as illustrated by Example \ref{ex:Solomyak} which is weakly-mixing.
Let us show it. 

\medskip 

\textbf{Back to Example \ref{ex:Solomyak}.} First  observe  the substitution $\gamma^{2}$ is proper. 
Assume $\exp(2i\pi \alpha)$ is an eigenvalue of the subshift generated by $\gamma$. From Lemma \ref{lemme:condcont}, $\alpha H(1)$ is of the form $t v + w$ with $t\in \RR$, $w \in \ZZ^{3}$  and $v$ is the eigenvector  $(1+\sqrt{2}, 2+ \sqrt{2}, 1)^{t}$ of  the incidence matrix $M_{\gamma}^{2}$. This implies that $\alpha$ is an integer
and therefore the associated subshift is weakly-mixing.  

\medskip 
 This example also shows that \eqref{hypoiii} does not implies  \eqref{hypoii}.

Conversely \eqref{hypoii} does not implies  \eqref{hypoiii}, as illustrated by the next example.

\begin{exemple}\label{ex:P1notP2}\rm
The eigenvalues of the incidence matrix $M_\rho$ of the substitution $\rho$ defined by 
$$
1\mapsto 1111112222, 2\mapsto 1122 ,
$$
are $4+2\sqrt{3}, 4-2\sqrt{3}$.
The normalized Perron right eigenvector of $M_\rho$ is $(1/\sqrt{3}, 1-1/\sqrt{3})^t$.
For $\mu$ the unique invariant measure of the subshift $(\Omega(\rho) , S)$ generated by $\rho$ it is classical that $\mu ([1]) = 1/\sqrt{3}$.
Moreover from standard linear computation and Proposition \ref{prop:condcont} (or see \cite[Theorem 1]{Adamczewski:2004}) one obtains that $\exp (2i\pi/\sqrt{3})$ is an eigenvalue of $(\Omega(\rho) , S)$ and thus that hypothesis \eqref{hypoii} is satisfied. 
But \eqref{hypoiii} is not satisfied as $\beta \vert \det M^{t}_{\rho\vert E^{s}} =(4+2\sqrt{3})( 4-2\sqrt{3})=4$.
\end{exemple}

We will show that the hypotheses \eqref{hypoii}, \eqref{hypoiii} are sufficient to ensure the mesurable conjugacy with an exchange of domain (Theorem \ref{theo:main2}), but they are not sufficient to prove the Pisot Conjecture as shown by the following example due to B. Sing \cite{Sing:2006}.

\begin{exemple}\label{ex:Sing}\rm
The substitution 
$$ 0 \mapsto 0 \bar{1}, 1 \mapsto 0, \bar{0}\mapsto \bar{0}1, \text{ and  } \bar{1} \mapsto \bar{0}$$ 
is also of weakly irreducible Pisot type.
Using the Host criterion for eigenvalues \cite{Host:1986} one can check the associated subshift $(\Omega_{\text{Sing}},S)$ has the same eigenvalues as the Fibonacci substitution subshift $\Omega_{\text{Fibo}}$. Recall that this system is measurably isomorphic to  the golden mean rotation $(\mathbb{T} , R)$ (see e.g. \cite{Fogg:2002}).
Thus  \eqref{hypoii} is satisfied and one easily check \eqref{hypoiii} also is. 
However $(\Omega_{\text{Sing}},S)$ cannot be measurably isomorphic to the golden mean rotation $(\mathbb{T} , R)$. 
Indeed, suppose it is.
The map $a\mapsto a, \bar{a} \mapsto a$, $a\in \{ 0,1\}$, defines a factor map from $(\Omega_{\text{Sing}},S)$ onto $(\Omega_{\text{Fibo}},S)$. This map is invariant under the automorphism $\varphi \colon \Omega_{\text{Sing}}  \to \Omega_{\text{Sing}} $ (or invertible cellular automata) given by the block map switching the barred letters with the non-barred ones. 
Thus, this factor defines an endomorphism of $(\mathbb{T} , R)$, i.e.,  a measurable map  $\delta$ of $(0,1)$ commuting with the golden mean rotation ($\delta \circ R = R\circ \delta$ a.e.). 
Since the map $x\mapsto \delta (x) - x \in (0,1)$ is $R$-invariant, by ergodicity it is a.e. constant. 
It follows $\delta$ is a.e.  invertible, and the map $\varphi$ is a.e.  constant. This is impossible since  $\varphi$ is an automorphism and  $S$ has no fixed point.

We conclude that $(\Omega_{\text{Sing}},S)$ is not measurably isomorphic to the golden mean rotation.
However we will see $(\Omega_{\text{Sing}},S)$ is measurably conjugate to a domain exchange (see Theorem \ref{theo:main2} below). 
\end{exemple}

\subsection{Domain exchange factorization}\label{sec:proof}
In this section we prove  the following theorem which specifies our main result Theorem \ref{theo:main}.
Let $\pi : \RR^{d_{\xi}} \to  \RR^{d_{\xi}}/\ZZ^{d_{\xi}} ={{\mathbb T}}^{d_{\xi}}$ denote the canonical projection.
We recall that $\Lambda_E$ is the measure induced by the Lebesgue measure on $E$.
It could be helpful to consider Figure \ref{fig:blabla} to resume the maps involved in the next theorem.

 \begin{theo}
\label{theo:main2}
Let $\xi$ be a substitution satisfying the hypothesis \eqref{hypoii} and \eqref{hypoiii}.  
Let $(\Omega, S)$ be the associated  substitution subshift. 
Then, 

there exist a self-affine domain exchange transformation $(E, \B, \lambda_E, T)$ in $\RR^{d_{\xi}}$   and a continuous onto  map 
$\Fetoile \colon \Omega \to E\subset \mathbb{R}^{d_{\xi}}$ which  is a measurable conjugacy map between the two systems.

 Moreover, there is a measurable map $\H \colon E \to \RR^{d_{\xi}}$ such that
\begin{itemize}
 
  	\item The map $\pi \circ \H\circ \Fetoile$ defines a (continuous) factor map  from $ (\Omega, S)$ to the dynamical system associated with  a minimal rotation on  the torus ${{\RR}}^{d_{\xi}}/\ZZ^{d_{\xi}}$.
	\item There is  a  constant $R_\xi \ge 1$ such that the map $\pi \circ \H\circ \Fetoile$ is a.e. $R_\xi$-to-one ({\it i.e.} $ \#(\pi \circ \H\circ \Fetoile)^{-1}(\{y\}) = R_{\xi}$ for a.e. $y \in \TT^{d_{\xi}}$  with respect to the Lebesgue measure). 
 \end{itemize}
\end{theo}
 Let us recall  Theorem \ref{theo:main2} applies to any unimodular Pisot substitution (see the discussion in Section \ref{sec:discuss}) and then provides Theorem  \ref{theo:main}.

We refer to Section \ref{sec:comments} for the motivation and interpretation of the introduction of the maps involved in  this theorem.

Let us comment that the statement and the proof of Theorem \ref{theo:main2}  is simplified when we assume that the 
substitution $\xi$ is  proper (see Definition \ref{def:subproper}) and its incidence matrix has a trivial kernel. 
Indeed the function $\H\circ \Fetoile - \Fetoile$ takes only integer vector values. 
In this case one has $\pi \circ \Fetoile = \pi \circ \H \circ \Fetoile$ is a (continuous) factor map onto a minimal rotation on the torus.

The strategy to prove Theorem \ref{theo:main2} is to reduce to a proper substitution thanks to  Corollary \ref{coro:wPisot}.
Then we use  the  approximations  of the eigenfunctions  given by Proposition \ref{prop:condcont}  to obtain  a projection map $\F$ from the subshift to an Euclidean space.
After checking this map has good properties with respect to the shift map and the substitution, we follow the same strategy as in  \cite{Canterini&Siegel:2001}.

\bigskip

We start by setting up the elements and observations necessary for the proof of Theorem \ref{theo:main2}. 
Let $\xi$ be a proper substitution satisfying Hypothesis \eqref{hypoii}.
 By Formula   (\ref{eq:rn}) on the entrance time $r_{n}$, we get for each $i\in \{1, \ldots, d_{\xi}\}$ and $x\in \Omega (\xi )$: 
$$\alpha_i r_n (x)   =  \sum_{k=1}^{n-1} \langle s_k (x),(M_\xi^{t})^{k-1} \alpha_i H (1) \rangle \mod \ZZ.$$
Notice this formula is also true for every power of $\xi$ instead of $\xi$.
Then,  with Formula (\ref{eq:wetv}) on the vectors $v({i})$, up to consider a power of $\xi$,  we obtain
\begin{align*}
\alpha_i r_n (x)  & = \sum_{k=1}^{n-1} \langle s_k (x),(M_\xi^{t})^{k-1} (v (i)+w(i))\rangle \mod \ZZ \\
  & =  \langle s_1 (x),  w(i)\rangle + \sum_{k=1}^{n-1} \langle s_k (x),(M_\xi^{t})^{k-1} v (i)\rangle \mod \ZZ .
\end{align*}

Let $\F_n  = \left(  (\langle s_1 (x),  w(i)\rangle \right)_{1\leq i\leq d_{\xi}}^t+ \Fetoile_n (x)$ where  
$$
\Fetoile_n  = \left(\sum_{k=1}^{n-1} \langle s_k ,(M_\xi^{t})^{k-1} v (i)\rangle  \right)^{t}_{1\leq i\leq d_{\xi}}.
$$
The  Proposition \ref{prop:condcont} and Lemma \ref{lemme:condcont} ensure  the sequence  $(\F_n)_{n\geq 1}$ uniformly converges to a continuous  function $\F \colon \Omega ({\xi} )\to \RR^{d_{\xi}}$,  explicitly  defined for $x \in \Omega ({\xi})$ by
$$
\F(x)  = \left(  
\langle s_1 (x),  w(i)\rangle 
+
\sum_{k=1}^{+\infty} \langle s_k(x) ,(M_\xi^t)^{k-1} v (i)\rangle
\right)^{t}_{1\leq i\leq d_{\xi}}.
$$

Let $V$ be the matrix with rows $v (1)^{t}, \ldots, v (d_{\xi})^{t}$. 
Then, the map $\F$ may be written as
\begin{align}\label{eq:defPhitilde}
 \F(x) =  \G(x)   + \Fetoile (x),
 \end{align}
where
\begin{align}\label{eq:defPhiGamma} 
 \Fetoile (x) = \sum_{k=1}^{+\infty}V M_\xi^{k-1} s_k(x) \quad  
\hbox{and }  \quad \G(x)  =(\langle s_1 (x),  w(i)\rangle)^{t}_{1\le i\le d_{\xi}}.
 \end{align} 
The function $\G$ is the difference between $\Fetoile$ and $\F$. As we will see in the next lemma, $\F$ satisfies the desired commutation property with the shift map  and the torus rotation (Item \eqref{lemma:conjrot}). Unfortunately it may not permute the action of the substitution with that of  a matrix  as 
$\Fetoile$  does (Item \eqref{lemma:conjsub}). As explained in Section \ref{sec:comments}, the map $\G$ will enable us to modify the projection mod $\ZZ^{d_\xi}$  by $\pi \circ \H$ to get a factorization onto a torus rotation with relevant commutation properties of the substitution.
In addition, the expression of $\G$ shows it is locally constant. This will be useful to describe  topological properties of the set $\F(\Omega(\xi))$ (Proposition \ref{prop:Topologie}) and to show the fiber of the factor map is constant (Proposition \ref{prop:cst-one}).  

\begin{lemma}
\label{lemma:deltaV} Let $\xi$ be a proper substitution satisfying \eqref{hypoii}. There exist  a continuous map $\Delta  \colon \Omega (\xi ) \to  \RR^{A_{\xi}} $ and a  bijective linear map $N \colon \RR^{d_{\xi}} \to \RR^{d_{\xi}}$ such that  for $\alpha = (\alpha_{1}, \ldots, \alpha_{d_{\xi}})^{t}$ and for any $x \in \Omega ({\xi})$,
\begin{enumerate}
\item 
\label{lemma:conjrot}
$\F\circ S  (x) = \F (x) + \alpha  \mod \ZZ^{d_{\xi}} $;
\item
\label{lemma:decompF}
$\Fetoile (x)  = V \Delta (x)$; 
\item
$M_\xi^t V^t = V^t N$;
\item
\label{lemma:conjugate}
the matrix  $N$  is conjugate to the matrix $M^{t}_{\xi\vert E^{s}}$ restricted to the space $E^{s}$;
\item 
\label{lemma:conjsub}
$\Fetoile\circ \xi (x)  = N^t ( \Fetoile (x))$.
\end{enumerate}
\end{lemma}

\begin{proof}
The eigenfunctions given by \eqref{hypoii}  and their approximations in Proposition \ref{prop:condcont} (see also Relation \eqref{eq:tpsretour}) provide 
$$\F\circ S  (x) = \F (x) + \alpha  \mod \ZZ^{d_{\xi}}.$$

Let us prove Statement \eqref{lemma:decompF}.  
We have 
\begin{eqnarray}
\label{eq:Fn}
\Fetoile_n(x) =  & V  \left(\sum_{k=1}^{n-1}  M_\xi^{k-1}s_{k}(x) \right) \\
=& V  {\rm Proj} \left(\sum_{k=1}^{n-1} M_\xi^{k-1}s_{k}(x) \right), 
\end{eqnarray}
where ${\rm Proj} \colon \RR^{A_{\xi}} \to E_{\xi}= \textrm{Vect }(v(1), \ldots, v(d_{\xi}))$ denotes the orthogonal projection onto $E_{\xi}$.
Recall that   by Corollary \ref{coro:ME}, the space $E_{\xi}$ has dimension $d_{\xi}$.
Since $(\Fetoile_n)_{n\geq 1}$ uniformly converges (see Proposition \ref{prop:condcont} and Lemma \ref{lemme:condcont}), the projection ${\rm Proj} (\sum_{k=1}^{n-1} M_\xi^{k-1}s_{k}(x))$ converges when $n$ goes to infinity to the  vector $\Delta(x)$ belonging to $E_{\xi}$ for any $x \in \Omega (\xi )$.
Therefore, we obtain Statement  \eqref{lemma:decompF}.

Let us prove  the other statements.
The basic properties of $s_{n} \circ \xi$ (Lemma \ref{lemma:lemmeclassic}) give  for any $x \in \Omega (\xi )$ and $n > 2$,
\begin{equation}\label{eq:Fnxi}
\Fetoile_n\circ \xi  =  V M_\xi \left(\sum_{k=1}^{n-2} M_\xi^{k-1}s_{k} \right).
\end{equation}

By the $\RR$-independence of the vectors $v (i)$ (Proposition \ref{prop:linind}), the linear map  $V^t \colon \RR^{d_{\xi}} \to E_{\xi}$ is bijective and since $M_\xi^{t} (E_{\xi}) = E_{\xi}$ (Corollary \ref{coro:ME}), there exists a bijective linear map $N \colon \RR^{d_{\xi}}  \to\RR^{d_{\xi}} $ such that

\begin{align}
\label{align:eigen}
M_\xi^{t} V^{t} = V^{t} N.
\end{align}

This shows  Statement (4).
Therefore, using  \eqref{eq:Fnxi}, \eqref{align:eigen} and \eqref{eq:Fn} we obtain for $n>2$,
$$
\Fetoile_{n}\circ \xi =  V M_\xi \sum_{k=1}^{n-2} M_\xi^{k-1}s_{k} =    N^{t }  \Fetoile_{n-1}.
$$

Letting $n$ going to infinity   we get \eqref{lemma:conjsub} and  complete the proof.
\end{proof}

The proof we present below follows the same scheme that was used and that appeared explicitly in \cite{Canterini&Siegel:2001} and in other works as \cite{Host:1992}.
The main difference here is that the factor map $\F$ is not built in the same way.  In the context of  \cite{Canterini&Siegel:2001} its function $\F$ itself satisfies the relation \eqref{lemma:conjsub} in Lemma \ref{lemma:deltaV} with the matrix $N$ being diagonal. 
Here we need to decompose the factor map $\F$ into  a sum $\Fetoile +\G$ where $\Fetoile$ satisfies \eqref{lemma:conjsub} in Lemma \ref{lemma:deltaV}.
Actually, such decomposition becomes useless when the incidence matrix of the  proper substitution $\xi$ has a trivial kernel because then the vectors $w(i)$ have integer coordinates and  $\G$ takes integer vector values. This simplifies the proof of the theorem significantly.

Recall that $\mu$ denotes the unique probability shift-invariant measure  of the system $(\Omega ({\xi}), S)$, and  $\lambda$ denotes the Lebesgue measure on ${\mathbb R}^{d_{\xi}}$.
As the map $\Fetoile$ is continuous with values in ${\mathbb R}^{d_{\xi}}$, it is usefull to observe that for every Borel set $B$ of $\Omega (\xi)$ the image $\Fetoile(B)$ is  analytic, hence is Lebesgue measurable in ${\mathbb R}^{d_{\xi}}$.   

\begin{lemma}
\label{lemma:perron}
Let $\xi$ be a proper substitution satisfying  \eqref{hypoii},\eqref{hypoiii}.  There exists a constant $C$ such that for every letter $a \in  A_{\xi}$ we have:

\begin{enumerate}
\item
\label{item:perronconstant}
$\lambda (\Fetoile ([a])) = C \mu ( [a]) $,
\item
\label{item:xinperronconstant}
for  all integer $n$ large enough, $\Fetoile  ([a])$ is the union of  the measure theoretically disjoint sets 
$$
\Fetoile(S^{-k}\xi^n ([b])) , \hbox{ with } 0\leq k< |\xi^n (b)| ,  [a] \cap S^{-k} \xi^n ([b])\not =  \emptyset,
$$

\item
\label{item:borelperronconstant}
for any Borel set $B \subset [a]$, 
$$
\lambda (\Fetoile(B)) = C \mu (B) .
$$
\end{enumerate}
\end{lemma}

\begin{proof}
The  Item \eqref{lemma:conjrot} of Lemma \ref{lemma:deltaV} implies $\Fetoile \circ S(x) - \Fetoile(x) = 
\G(x) - \G(S(x)) + \alpha \mod \ZZ^{d_{\xi}}$, where $\G$ is defined in \eqref{eq:defPhiGamma}.
Since the function $s_{1}$ takes finitely many values and $\Fetoile$ is bounded, $\H_{diff}= \Fetoile \circ S  - \Fetoile$  is a continuous function taking values into a finite set. 
So the function $\H_{diff}$  is locally constant.
Hence, there exists  some integer $n_{0}\ge 0$ such that  $\H_{diff}$ is constant on each set $S^{-k} \xi^n([b])$, with $n > n_{0}$, $b\in A_{\xi}$  and  $0\leq k< |\xi^n (b)|$ (see Proposition \ref{substKR}).

Therefore, from Lemma \ref{lemma:lemmeclassic} and Item \eqref{lemma:conjsub} of Lemma \ref{lemma:deltaV}, for any such $b$ and $k$, there exists a vector  $\delta(k,b)\in \RR^{d_{\xi}}$ such that 
$$ 
\Fetoile(S^{-k} \xi^n ([b]))  = \delta (k,b) + \Fetoile (\xi^n ([b])) 
 = \delta (k,b) +(N^t)^n \Fetoile ([b]).
$$

By the very hypothesis $\eqref{hypoiii}$ and Item \eqref{lemma:conjugate} of Lemma \ref{lemma:deltaV}, we have  $\vert \det N^t \vert =1/ \beta$, so we get
\begin{align*}
\lambda ( \Fetoile(S^{-k} \xi^n ([b])) ) & = \lambda ( (N^t)^n \Fetoile ([b])) = \vert \det (N^t)^n \vert \lambda (\Fetoile([b])) \\
& = \frac{1}{\beta^n} \lambda (\Fetoile([b])).
\end{align*}

Let $a\in A_{\xi}$, the partitions of $\Omega ({\xi})$ in Proposition \ref{substKR} provide
$$
[a]  = \bigcup_{\substack{0\leq k< |\xi (j)|, b\in A_{\xi} \\ [a] \cap S^{-k} \xi^n ([b])\not = \emptyset}} S^{-k} \xi^n([b]) .
$$

Consequently,
\begin{align}
\label{align:perronequality}
\lambda ( \Fetoile ([a])) \leq &  \sum_{\substack{k,b; 0\leq k< |\xi^n (b)| ,\\   [a] \cap S^{-k} \xi^n ([b])\not =  \emptyset}} \frac{1}{\beta^n} \lambda (\Fetoile([b])) 
= \frac{1}{\beta^n }(M_\xi^n (\lambda ( \Fetoile ([b])))^t_{b\in A_{\xi}})_a.
\end{align}

From the  Perron's Theorem, the above inequality is an equality and $(\lambda ( \Fetoile( [b])))^t_{b\in A_{\xi}}$  is a multiple of the eigenvector $( \mu ([a]))^t_{a\in A_{\xi}} = \vec{\mu}(1)$ of the dominant eigenvalue $\beta^n$ of $M_{\xi}^n$. 
This shows Item  \eqref{item:perronconstant}. Notice that  the equality in \eqref{align:perronequality} also implies Item \eqref{item:xinperronconstant}.

To prove Item \eqref{item:borelperronconstant}, it is enough  to use the partitions of $\Omega ({\xi})$ given in Proposition \ref{substKR} and the ideas at the beginning of this proof. 
We left it to the reader (see the proof  of Proposition 4.3 in \cite{Canterini&Siegel:2001}).
\end{proof}

With the next two propositions, we continue to follow the approach (and the proofs) in \cite{Canterini&Siegel:2001}.

\begin{prop}
\label{prop:injcyl}
Let $\xi$ be a proper substitution satisfying  \eqref{hypoii},\eqref{hypoiii}.
The map $\Fetoile$ is  $\mu$-a.e.  one-to-one on each cylinder set $[a]$:
there exists a $\mu$-negligeable borelian  subset $\mathcal{M} \subset \Omega (\xi )$ such that  for any $x$ and $y$ in $[a]\setminus \mathcal{M}$ satisfying $\Fetoile (x) = \Fetoile(y)$, we have $x=y$.  
\end{prop}

\begin{proof}
Let $a\in A_{\xi}$. From Lemma \ref{lemma:perron}, the sets 
$$
\mathcal{N}_a^{(\ell)} = 
\bigcup_{\substack{(k_1,j_1)\not = (k_2,j_2) ;\\ 
0\leq k_1< |\xi^\ell (b_1)|,  [a] \cap S^{-k_1} \xi^\ell  ([b_1])\not = \emptyset \\
0\leq k_2< |\xi^\ell (b_2)|,  [a] \cap S^{-k_2} \xi^\ell ([b_2])\not = \emptyset}} 
\Fetoile(S^{-k_1} \xi^\ell([b_1])) \cap \Fetoile(S^{-k_2} \xi^\ell([b_2]))
$$

have zero $\lambda$-measure, for any $\ell\in \mathbb{N}$ large enough. Item \eqref{item:borelperronconstant} of Lemma \ref{lemma:perron} gives furthermore that the sets $\mathcal{M}_a^{(\ell)} = (\Fetoile)^{-1} (\mathcal{N}_a^{(\ell)})$ have zero measure with respect to $\mu$.

Let $x_1$ and $x_2$ be two distinct elements of $[a]$ such that $\Fetoile(x_1) = \Fetoile(x_2)$.
It suffices to show that they belong to some $ \mathcal{M}_a^{(\ell)}$.
Considering the partitions $\P_{\ell}$, $\ell\ge 0$, of Proposition \ref{substKR}, there exist infinitely many  $\ell \in \mathbb{N}$  with two distinct couples $(k_1, b_1)$ and $(k_2,b_2)$, such that $0\leq k_1< |\xi^\ell (b_1)|$, $0\leq k_2< |\xi^\ell (b_2)|$, $x_1 \in S^{-k_1} \xi^\ell ([b_1])$ and $x_2 \in S^{-k_2} \xi^\ell ([b_2])$.
Then, $x_1$ and $x_2$ belong to $\mathcal{M}_a^{(\ell)}$ for infinitely many $\ell$,   which achieves the proof.
\end{proof}

\begin{prop}
\label{prop:Finj}
Assume Hypotheses \eqref{hypoii},\eqref{hypoiii}  for a proper substitution $\xi$.
The sets $\mathcal M$ and $\mathcal N$ are negligeable  borelian sets and the map $\Fetoile: \Omega(\xi)\setminus {\mathcal M} \to \Fetoile(\Omega(\xi))\setminus {\mathcal N}$ is a bi-mesurable map.
\end{prop}

\begin{proof}
As  $\xi$ is proper, there exists a letter $a$ such that $\xi (\Omega (\xi ) )$ is included in $[a]$.
Therefore, from Proposition \ref{prop:injcyl}, the map $\Fetoile$ is one-to-one on $\xi (\Omega (\xi ) )$ except on a set $ \mathcal{M}$ of zero measure. 
By the basic properties of the map $\Fetoile$  (precisely Item \eqref{lemma:conjsub} of Lemma \ref{lemma:deltaV}), if two points $x,y \in \Omega ({\xi})$ have the same image through $\Fetoile$, then  $\Fetoile(\xi(x))= \Fetoile(\xi(y))$, and hence $x$ and $y$ belong to  $\xi^{-1}({\mathcal{M}})$.   

Recall that the induced system on $\xi(\Omega ({\xi}))$ is a factor of $(\Omega ({\xi}), S)$ via the map $\xi$ (see Relation \eqref{eq:sysinduit}). This implies that the measure $\mu (\xi^{-1}(\cdot))$  is invariant for the induced system $(\xi(\Omega ({\xi})),  S_{\xi(\Omega ({\xi}))})$. 
It is a classical fact that induced dynamical systems of uniquely ergodic dynamical systems are uniquely ergodic. 
Thus, $(\xi(\Omega ({\xi})),  S_{\xi(\Omega ({\xi}))})$ has a unique invariant probability measure: $\mu /\mu  (\xi(\Omega ({\xi})))$.  
Thus $\mu(\xi^{-1}({\mathcal{M}}))$ is proportional to $\mu({\mathcal M})$, so it is null. This achieves the proof. 
\end{proof}

 The following proposition provides some topological properties of the set $\phi(\Omega(\xi))$. 
 Its proof is a modification of the arguments in \cite{Kulesza:1995} Lemma 2.1. 
\begin{prop}\label{prop:Topologie} Assume Hypotheses \eqref{hypoii} for a proper substitution $\xi$.
\label{prop:Regul}
 For any clopen set $C$ in $\Omega ({\xi})$, the set $\Fetoile (C)$ is regular, {i.e.}, $$\overline{{\rm int \ } \Fetoile (C)} = \Fetoile (C),$$
  where ${\rm int\  } U$ denotes the interior of  the set $U$ for the usual  Euclidean topology.
\end{prop}
\begin{proof} First we show that the open set ${\rm int \ } \Fetoile (\Omega ({\xi}))$ is not empty.  Since  $1,  \alpha_{1}, \ldots, \alpha_{d_{\xi}}$ are rationally independent, by Lemma \ref{lemma:deltaV},  denoting by   $\pi$  the canonical projection $\RR^{d_{\xi}} \to  \RR^{d_{\xi}}/\ZZ^{d_{\xi}} = \TT^{d_{\xi}}$, the  continuous map  $\pi \circ \F  \colon \Omega ({\xi}) \to \TT^{d_{\xi}}$ has a dense image hence is onto.
It follows  that for any  small $\epsilon >0$, there exists a finite family $\V$  of integer vectors  such that  
$$ 
B_{\epsilon}(0) \subset   \bigcup_{p \in \V} \F(\Omega ({\xi}))+p.
$$
As in the proof of Lemma \ref{lemma:perron} we consider the function $\G: \Omega(\xi) \to \mathbb{R}^{d_{\xi}}$ defined by $ \G( x) = (\langle s_1 (x),  w(i)\rangle)_{1\le i\le d_{\xi}}$ and recall that $\F(x) = \G(x) +\Fetoile (x)$. Set $ \mathcal G$ to be the finite set of values of the map $\G$, thus we obtain
$$ B_{\epsilon}(0) \subset   \bigcup_{p \in \V} \F(\Omega ({\xi}))+p \subset   \bigcup_{g \in {\mathcal G}, p \in \V} \Fetoile (\Omega ({\xi}))+p +g.$$
By the Baire Category Theorem, the set $\Fetoile (\Omega ({\xi}))$ has a nonempty interior.

Now let $\Omega^* =\Omega ({\xi}) \setminus U$ where $U$ is the union of all open sets $O$ such that ${\rm int \ }\Fetoile (O) = \emptyset$.
From the previous remark it is a nonempty compact set.  Notice that $\Omega ({\xi}) \setminus \Omega^* $ is the union of countably many open (and then $\sigma$-compact) subsets. The image $\Fetoile (\Omega ({\xi}) \setminus \Omega^* )$ is then a countable union of compact sets each of those with an empty interior. Thus, again by the Baire Category Theorem, the set $\Fetoile (\Omega^*)$ is dense in $\Fetoile (\Omega ({\xi}))$ and since $\Omega^*$ is compact, $\Fetoile (\Omega^*)= \Fetoile (\Omega ({\xi}))$.

Let us show that $\Omega^*$ is $S$-invariant. 
Let $O$ be an open set in $\Omega ({\xi})$ such that ${\rm int\  }\Fetoile (O)$ is empty. 
By Lemma  \ref{lemma:deltaV}, the function $\Fetoile \circ S - \Fetoile $ takes finitely many values, hence it is constant on a partition by clopen sets $\P$  of $\Omega ({\xi})$.  
For any  atom $C$ of $\P$ the set $ {\rm int\ }\Fetoile (C\cap O)$ is empty  and  then   $ {\rm int \ }\Fetoile (S(C\cap O))$ is also empty. 
As $\Fetoile (S O) = \cup_{C\in \P} \Fetoile (S (C \cap O))$, the set $\Fetoile (S O)$ is a countable union of compact sets with empty interiors.  Again by the Baire Category Theorem, the set $\Fetoile (SO)$ has empty interior, and $\Omega^*$ is $S$-invariant.
   
The  minimality implies that $\Omega^* = \Omega ({\xi})$, so  the image by $\Fetoile $  of any open set has a nonempty interior. 

Finally, let $C$ be  a clopen set, and assume that $W := \Fetoile (C) \setminus \overline{{\rm int \ } \Fetoile (C)}$ is not empty. From the previous assertion, the set $\Fetoile ( \Fetoile^{-1}(W) \cap C) = U$ contains a ball and then $W$ intersects ${\rm int\  }\Fetoile (C)$: a contradiction. This shows the statement of the proposition. 
\end{proof}


Recall the  map $\G(x) =  (\langle s_1 (x),  w(i)\rangle)_{1\le i\le d_{\xi}}$, defined in \eqref{eq:defPhiGamma},  is constant on the sets  
$S^{-k}\xi ([a])$, $0\leq k< |\xi (a)|$, $a\in A_{\xi}$.
From Proposition \ref{prop:injcyl} one can define a map $\H : \Phi (\Omega (\xi )) \to \mathbb{R}^{d_{\xi}}$ almost everywhere by 
\begin{align}\label{eq:defPsi}
\H (y) = y +(\langle s_1 (x),  w(i)\rangle)_{1\le i\le d_{\xi}} \hbox{ where } y= \Phi (x).
\end{align}
 It follows it satisfies 
\begin{align}\label{eq:relPhiGamma} 
\H\circ \Fetoile = \Phi+ \G = \F, \quad \mu\text{-a.e.}.
\end{align}
Recall that  $\pi : \RR^{d_{\xi}} \to  \RR^{d_{\xi}}/\ZZ^{d_{\xi}} ={{\mathbb T}}^{d_{\xi}}$ denotes the canonical projection.

In the next proposition we use the fact that $\H\circ \Fetoile-\Fetoile$ is locally constant.

\begin{prop}\label{prop:cst-one} Assume Hypotheses \eqref{hypoii}, \eqref{hypoiii} for a proper substitution $\xi$.
There exists a uniformly discrete set $\Lambda$  such that for $\lambda$-a.e. $y \in \Fetoile(\Omega(\xi)),$ 
$$\{y_1- y_2;  \pi \circ \H  (y_1) = \pi \circ \H  (y_2)= y \} \subset \Lambda.$$ 
Moreover the maps $Z: \Omega ({\xi}) \to \mathbb{Z}\cup\{\infty\}$ and $Z_\H \colon \Fetoile (\Omega(\xi)) \to \mathbb{Z}\cup\{\infty\}$ defined by 
\begin{align*}
Z(x) &= \# (\pi \circ \H \circ \Fetoile )^{-1} (\{ \pi \circ \H \circ \Fetoile  (x) \} )  \hfill \text{ and } \\
Z_\H(y) &= \# (\pi \circ \H )^{-1} (\{ \pi \circ \H  (y) \} )
\end{align*}
 are finite and constant  $\mu$-a.e. and $\lambda$-a.e with the same constant.
\end{prop}
\begin{proof}
For any $z\in \mathbb{Z}^{d_{\xi}}$, set  $B_{z}$ to  be the set $\{ y \in \Fetoile (\Omega(\xi)); \ \exists y' \in \Fetoile(\Omega ({\xi})), y = y' + z \}$.
We have $B_z =  \Fetoile(\Omega ({\xi})) \cap (\Fetoile(\Omega ({\xi})) + z)$, so it is a Borel set.  
Notice that for any integer $n$, $Z_\H^{-1} (\{ n \} )$ is a finite intersection of such sets, so the map $Z_\H$ is measurable.
Since $Z = Z_\H \circ \Fetoile$ a.e., the map $Z$ is also measurable. 
Recall that  the map $\Fetoile $ is a.e.  one-to-one (Proposition \ref{prop:Finj})
and  $\H\circ \Fetoile  - \Fetoile =\G$ (see Equation \eqref{eq:relPhiGamma}).  
Recall   the map $\G$ takes values in  the finite set ${\mathcal H} = \{ (\langle s_1 (x),  w(i)\rangle)^{t}_{1\le i\le d_{\xi}}; x \in \Omega(\xi))\}$. 

Hence for two elements $\Fetoile(x)$ and $\Fetoile(x')$, with $x,x' \in \Omega(\xi)$ in the same $\pi \circ \H$ fiber ({\it i.e.}  $\pi \circ \H \circ \Fetoile(x) = \pi \circ \H\circ \Fetoile(x')$), the difference $\Fetoile(x)-\Fetoile(x')$ a.s. belongs to ${\mathbb Z}^{d_{\xi}}+ {\mathcal H} - {\mathcal H}$, a uniformly discrete set. This proves the first claim for $\Lambda = {\mathbb Z}^{d_{\xi}}+ {\mathcal H} - {\mathcal H}$.

It follows, by the compactness of the set $\Fetoile(\Omega(\xi))$, that the map $Z_\H$  is  a.e.  finite.  
Hence it is also the case for the map $Z$.
Moreover since it  is $S$-invariant, we conclude by ergodicity that the map $Z$ is a.e. constant.
Since $Z = Z_\H \circ \Fetoile$, the map $Z_\H$ is a.e. constant with the same constant.
\end{proof}

\begin{proof}[Proof of Theorem \ref{theo:main2}] 
Let $\xi \colon A_{\xi} \to A_{\xi}^{*}$ be a substitution satisfying the hypothesis \eqref{hypoii} and \eqref{hypoiii}.
By Corollary \ref{coro:wPisot}, we can moreover assume that $\xi$ is proper. 
So we can define the maps $\Fetoile, \F,\G, \H$ (see  \eqref{eq:defPhitilde}, \eqref{eq:defPhiGamma} and \eqref{eq:defPsi})    satisfying all the former properties proven in Section \ref{sec:proof}.

Let $E$ be the compact  set  $\Fetoile (\Omega ({\xi}))$. Since $\G$ takes finitely many values,   Lemma \ref{lemma:deltaV}, Relation \eqref{eq:defPhitilde}  and Proposition \ref{substKR},  ensure the existence of an integer $n$  such that  the map $\Fetoile \circ S - \Fetoile  $ is constant on each set $E_{n,a,k,} :=\Fetoile (S^{-k}\xi^{n}([a]))$ with $a \in A_{\xi}$ and $0 \le k < \vert \xi^{n}(a) \vert.$ Let $T$ be the transformation such that its restriction  on $E_{n,a,k}$ is the translation of the vector $(\Fetoile \circ S -\Fetoile )|_{E_{n,a,k}}$. Lemma \ref{lemma:perron}  ensures it is well defined almost everywhere on $E$ with respect to the Lebesgue measure $\lambda$. We have to show that this map defines a domain exchange. 

First, notice, from Proposition \ref{prop:Regul}, that each compact set $E_{n,a,k}$ is regular, hence has a positive Lebesgue measure. 
From Lemma \ref{lemma:perron}, the sets $(E_{n,a,k})_{a,k}$ are pairwise disjoint in measure for the measure $\lambda$. 
Observe, moreover, from the definition of $T$,  that  it   satisfies $\Fetoile \circ S = T \circ \Fetoile $ $\mu$-a.e.  Hence we get $\lambda(T(E)) = \lambda(E)$ and $T(E) \subset E$ modulo $\lambda$-null sets. Thus, it remains to show  that $T$ preserves the Lebesgue measure $\lambda$.

By Proposition \ref{prop:Finj} the map  $T$ is one-to-one except on a set of zero measure. As a consequence, the sets $(T(E_{n,a,k}))_{a,k}$ are also disjoint in measure. Furthermore, the map $T^{-1}$ is defined $\lambda$-a.e. and is a translation on a subset of full $\lambda$-measure of $T(E_{n,a,k})$. 
The family of sets $(E_{n,a,k})_{n,a,k}$ generates the Borel $\sigma$-algebra (propositions \ref{substKR} and \ref{prop:Finj}).  
Lemma \ref{lemma:perron} ensures that the map $T$ preserves the Lebesgue measure and $T$ defines a domain exchange transformation. 

Let us also mention that Item \eqref{lemma:conjsub} of Lemma \ref{lemma:deltaV} provides  this domain exchange is self-affine with respect to the sets $E_{n,a,k}$ and the linear part $(N^{t})^{n}$.  Finally, Proposition \ref{prop:Finj} shows this domain exchange is measurably conjugate to the subshift $(\Omega (\xi ), S)$.

We recall $\H \colon E \to \F(\Omega({\xi})) \subset \RR^{d_{\xi}}$ is the measurable  map defined by $\H\circ \Fetoile = \F$ $\mu$-a.e.. 
Proposition \ref{prop:cst-one} gives the map $\pi \circ \H \circ \Fetoile  \colon \Omega ({\xi}) \to \TT^{d_{\xi}}$ is $\lambda$-a.e. $R_\xi$-to-one for some integer $R_\xi$ and  is $\mu$-a.e. equal to $\pi \circ \F$ which is a (continuous) factor map onto a  minimal rotation on the torus $\TT^{d_{\xi}}$.
Hence $\pi \circ \H \circ \Fetoile $ is a measurable  a.e. $R_\xi$-to-one factor map from $(\Omega (\xi ), S)$ onto a minimal rotation on $\TT^{d_{\xi}}$. 
\end{proof}

\begin{proof}[Proof of Theorem \ref{theo:main}] 
As,  by Corollary \ref{cor:Pisotsub} and Proposition  \ref{prop:vpPisot}, unimodular Pisot substitutions satisfy the hypothesis \eqref{hypoii}, \eqref{hypoiii}, Theorem \ref{theo:main} is a direct consequence of Theorem \ref{theo:main2}.
\end{proof}

\subsection{ On the cardinality of the fibers}

 Proposition \ref{prop:cst-one} tells that the map $\pi \circ \Psi$ is a.s. $R_{\xi}$-to-one for some constant $R_{\xi}$. We show below that $R_{\xi} $ is the volume $\lambda (\Fetoile(\Omega ({\xi})))$ of the set $\Fetoile(\Omega ({\xi}))$.

\begin{prop} Assume Hypotheses \eqref{hypoii}-\eqref{hypoiii}.
We have $$\lambda(\Fetoile(\Omega ({\xi}))) = R_\xi .$$
\end{prop}
\begin{proof}
Recall that  the map $\Psi$ satisfies  $\Psi \circ \Phi = \tilde{\Phi }= \Phi + \G$ where $\G$ is defined in  \eqref{eq:defPhiGamma}, so that $\Psi$ restricted to a set $\Phi(S^{-k} \xi^n([a]))$, $a \in A_{\xi}$, $0 \le k < |\xi(a)|$, for $n$ large enough, is defined on a subset $E'_{n,a,k}$ of full measure and  is  a.s.   the translation by the constant value  $\G|_{S^{-k} \xi^n([a])}$.  
Thus the closure of $\Psi(E'_{n,a,k})$ is a translated of $\Phi(S^{-k} \xi^n([a]))$ and $\Psi$ admits a local continuous right inverse: a translation  $\Psi^{-1}_{E'_{n,a,k}} \colon \overline{\Psi(E'_{n,a,k})}  \to  \Phi(S^{-k} \xi^n([a]))$  such that $\Psi\circ \Psi^{-1}_{E'_{n,a,k}} = {\rm Id}$ a.s..

Since the map $\pi$ is open, by Proposition \ref{prop:Regul}, the sets $\pi(\overline{ \Psi (E'_{n,a,k})})$ are regular. 
Moreover, taking $n$ large enough, we can assume that the diameter of each set $\overline{\Psi (E'_{n,a,k})}$ is smaller than $1$ so that $\pi$ restricted to ${\overline{\Psi (E'_{n,a,k})}}$ is one-to-one. 

Let $I$ be a maximal (for the inclusion)   subset  of $\{ E'_{n,a,k}; a \in A_{\xi}, 0 \le k < |\xi(a)| \}$ such that $V=\bigcap_{E \in I} {\rm int}(\pi( \overline{\Psi (E)}))$ is not empty.
Since the sets are regular, if a set  $\pi (\overline{\Psi (E'_{n,a,k})}) \subset \TT^{d_{\xi}}$ intersects an open set  $O$, then the interior of $\pi (\overline{\Psi (E'_{n,a,k})})$ also intersects $O$.
It follows  from the maximality that: 
\begin{itemize}
\item[($\texttt{H}_{1}$)]\label{it:1}  The open set $(\pi \circ \Psi)^{-1}(V)$   intersects no set $\overline{E'_{n,a,k}}$   where  $E'_{n,a,k}$  does not belong to $I$.

\end{itemize}
 
So, the map $\pi$ restricted to ${\overline{\Psi (E'_{n,a,k})}}$ being one-to-one, almost every element in $V$ admits  $\# I$ pre-images in $\cup_{a,k} E'_{n,a,k}$ with respect to the  map $\pi \circ \Psi$. 
Since $V$ has a positive Lebesgue measure, Proposition  \ref{prop:cst-one} ensures $\# I = R_{\xi}$. 

Moreover, by the definition of $V$,  for any open subset $U \subset ]0,1[^{d_{\xi}}$ such that $\pi(U) \subset V$, one has 
\begin{itemize}
\item[($\texttt{H}_{2}$)]\label{it:2} for any set $E \in I$, there exists an integer vector $z \in \ZZ^{d_{\xi}}$ such that $U+z \subset \overline{\Psi(E)}$.
\end{itemize}

Denoting  $\lambda_{\TT^{d_{\xi}}}$  the Lebesgue measure on the torus $\TT^{d_{\xi}}$, it follows 

$$
\begin{array}{ll}
\lambda(U) &=  \lambda_{\TT^{d_{\xi}}}(\pi(U)) \\
&= \frac 1{\lambda((\Phi(\Omega(\xi)))}\lambda((\pi \circ \Psi)^{-1}(\pi(U))) \\
&=\frac 1{\lambda(\Phi(\Omega(\xi)))}\lambda(\Psi^{-1}(\bigcup_{E \in I, z \in \ZZ^{d_{\xi}}} \overline{\Psi(E)} \cap  (U+z) ) ) \\
&= \frac 1{\lambda(\Phi(\Omega(\xi)))}\lambda(\bigcup_{E  \in I, z \in \ZZ^{d_{\xi}}} \Psi^{-1}_{E}(U+z)\cap E) \\
&= \frac{R_{\xi}}{\lambda(\Phi(\Omega(\xi)))} \lambda(U).
\end{array}
$$
The first equality comes from properties of the Lebesgue measure on the torus. The second 
one comes from the fact that, by unique ergodicity, the pullback image of the Lebesgue measure on the torus  by the measure theoretical factor map $\pi \circ \Psi$ is the normalized measure. 
The third equality follows from  ($\texttt{H}_{1}$) and ($\texttt{H}_{2}$). The properties of the maps $\Psi_{E}^{-1}$ provide the fourth equality. The last one follows from the facts the sets $E'_{n,a,k}$ form a measurable partition of $\Phi(\Omega(\xi))$, the maps $\Psi_{E}^{-1}$ are translations and $\# I = R_{\xi}$. This provides the conclusion.
\end{proof}

The function $\pi \circ \F \colon \Omega(\xi) \to \TT^{d_{\xi}}$ is a continuous factor map. We provide here a topological characterization of the constant $R_\xi$.
\begin{prop}\label{prop:min} Assume Hypotheses \eqref{hypoii}-\eqref{hypoiii}.
We have 
$$R_\xi= \inf_{x \in \Omega(\xi)} \#(\pi \circ \F )^{-1} (\{\pi  \circ \F (x)\}).$$
\end{prop}
\begin{proof} 
Let $\tilde{\Omega}(\xi) \subset  \Omega(\xi)$ be a $S$-invariant set of full measure such that the map $ x \mapsto \#(\pi \circ \H)^{-1} (\{\pi  \circ \F (x)\})$ is constant, equal to $R_\xi$ and any $\pi \circ \H$ fiber is uniformly discrete (Proposition \ref{prop:cst-one}).  
Let $x \in\tilde{\Omega}(\xi)$ and for any integer $n$, let $y_1^{(n)}, \ldots, y_{R_\xi}^{(n)} \in \Fetoile(\Omega(\xi))$ be $R_\xi$  elements  in $(\pi \circ \H)^{-1} (\{\pi  \circ \F (S^n x)\})$ and 
let $x_1^{(n)}, \ldots, x_{R_\xi}^{(n)} \in \Omega(\xi)$ be their pre-image by $\Fetoile$: $\Fetoile(x_i^{(n)}) = y_i^{(n)}$.
Since  the set of $y_j^{(n)}$ is uniformly discrete and the map $\Fetoile$ is uniformly continuous,  the set of  points $\{ x_j^{(n)}; j = 1, \ldots,R_\xi\}$ is uniformly (in $n$) separated. 

Let $x_0 \in \Omega(\xi)$ be such that $\#(\pi \circ \F)^{-1} (\{\pi  \circ \F (x_0)\}) = \inf_{x \in \Omega(\xi)} \#(\pi \circ \F)^{-1} (\{\pi  \circ \F (x)\})$. 
Taking a sequence of integer $(n_i)_i$ such that $S^{n_i} x$ goes to $x_0$ when $i$ goes to infinity,  and taking  accumulation points of  $x_j^{(n_i)}$, provides $R_\xi$  different points in $(\pi \circ \F)^{-1} (\{\pi  \circ \F (x_0)\})$. This proves the claim.
\end{proof}
As a consequence of Proposition  \ref{prop:min}, to prove the Pisot conjecture it is enough to show the existence  of  a point in the torus that admits only one preimage by the  continuous map $ \pi \circ \F$. A similar observation was done in \cite{Baker&Barge&Kwapisz:2006}.

\bibliography{pisotjuillet2021}

\begin{thebibliography}{10}

\bibitem{Adamczewski:2004}
B.~Adamczewski.
\newblock Symbolic discrepancy and self-similar dynamics.
\newblock {\em Ann. Inst. Fourier}, 54:2201--2234, 2004.

\bibitem{ABBLS:2015}
S.~Akiyama, M.~Barge, V.~Berth\'{e}, J.-Y. Lee, and A.~Siegel.
\newblock On the {P}isot substitution conjecture.
\newblock In {\em Mathematics of aperiodic order}, volume 309 of {\em Progr.
  Math.}, pages 33--72. Birkh\"{a}user/Springer, Basel, 2015.

\bibitem{Arnoux&Ito:2001}
P.~Arnoux and S.~Ito.
\newblock Pisot substitutions and {R}auzy fractals.
\newblock {\em Bull. Belg. Math. Soc. Simon Stevin}, 8:181--207, 2001.
\newblock Journ{\'e}es Montoises d'Informatique Th{\'e}orique
  (Marne-la-Vall{\'e}e, 2000).

\bibitem{Arnoux&Rauzy:1991}
P.~Arnoux and G.~Rauzy.
\newblock Repr\'esentation g\'eom\'etrique de suites de complexit\'e {$2n+1$}.
\newblock {\em Bull. Soc. Math. France}, 119:199--215, 1991.

\bibitem{Baker&Barge&Kwapisz:2006}
V.~Baker, M.~Barge, and Jaroslaw J.~Kwapisz.
\newblock Geometric realization and coincidence for reducible non-unimodular
  {P}isot tiling spaces with an application to {$\beta$}-shifts.
\newblock {\em Ann. Inst. Fourier (Grenoble)}, 56:2213--2248, 2006.
\newblock Num{\'e}ration, pavages, substitutions.

\bibitem{Barge:2016}
M.~Barge.
\newblock Pure discrete spectrum for a class of one-dimensional substitution
  tiling systems.
\newblock {\em Discrete Contin. Dyn. Syst.}, 36:1159--1173, 2016.

\bibitem{Barge:2018}
M.~Barge.
\newblock The {P}isot conjecture for {$\beta$}-substitutions.
\newblock {\em Ergodic Theory Dynam. Systems}, 38:444--472, 2018.

\bibitem{Barge&Diamond:2002}
M.~Barge and B.~Diamond.
\newblock Coincidence for substitutions of {P}isot type.
\newblock {\em Bull. Soc. Math. France}, 130:619--626, 2002.

\bibitem{Barge&Kwapisz:2006}
M.~Barge and J.~Kwapisz.
\newblock Geometric theory of unimodular {P}isot substitutions.
\newblock {\em Amer. J. Math.}, 128:1219--1282, 2006.

\bibitem{Berthe&Jolivet&Siegel:2012}
V.~Berth{\'e}, T.~Jolivet, and A.~Siegel.
\newblock Substitutive {A}rnoux-{R}auzy sequences have pure discrete spectrum.
\newblock {\em Unif. Distrib. Theory}, 7:173--197, 2012.

\bibitem{Bressaud&Durand&Maass:2005}
X.~Bressaud, F.~Durand, and A.~Maass.
\newblock Necessary and sufficient conditions to be an eigenvalue for linearly
  recurrent dynamical cantor systems.
\newblock {\em J. London Math. Soc.}, 72:799--816, 2005.

\bibitem{Bressaud&Durand&Maass:2010}
X.~Bressaud, F.~Durand, and A.~Maass.
\newblock On the eigenvalues of finite rank {B}ratteli-{V}ershik dynamical
  systems.
\newblock {\em Ergodic Theory Dynam. Systems}, 30:639--664, 2010.

\bibitem{Canterini&Siegel:2001}
V.~Canterini and A.~Siegel.
\newblock Geometric representation of substitutions of pisot type.
\newblock {\em Trans. Amer. Math. Soc.}, 353:5121--5144, 2001.

\bibitem{Dekking:1978}
F.~M. Dekking.
\newblock The spectrum of dynamical systems arising from substitutions of
  constant length.
\newblock {\em Z. Wahrscheinlichkeitstheorie und Verw. Gebiete}, 41:221--239,
  1978.

\bibitem{Durand:1996}
F.~Durand.
\newblock {\em Contributions \`a l'\'etude des suites et syst\`emes dynamiques
  substitutifs}.
\newblock PhD thesis, Universit\'e de la M\'editerran\'ee (Aix-Marseille II),
  1996.

\bibitem{Durand:1998a}
F.~Durand.
\newblock A characterization of substitutive sequences using return words.
\newblock {\em Discrete Math.}, 179:89--101, 1998.

\bibitem{Durand:1998b}
F.~Durand.
\newblock A generalization of {Cobham's} theorem.
\newblock {\em Theory Comput. Systems}, 31:169--185, 1998.

\bibitem{Durand:1998c}
F.~Durand.
\newblock Sur les ensembles d'entiers reconnaissables.
\newblock {\em J. Th\'eor. Nombres Bordeaux}, 10:65--84, 1998.

\bibitem{Durand&Host&Skau:1999}
F.~Durand, B.~Host, and C.~Skau.
\newblock Substitutive dynamical systems, bratteli diagrams and dimension
  groups.
\newblock {\em Ergod. Th. \& Dynam. Sys.}, 19:953--993, 1999.

\bibitem{EiIto:2005}
Hiromi Ei and Shunji Ito.
\newblock Tilings from some non-irreducible, {P}isot substitutions.
\newblock {\em Discrete Math. Theor. Comput. Sci.}, 7(1):81--121, 2005.

\bibitem{Ferenczi&Mauduit&Nogueira:1996}
S.~Ferenczi, C.~Mauduit, and A.~Nogueira.
\newblock Substitution dynamical systems: algebraic characterization of
  eigenvalues.
\newblock {\em Ann. Sci. \'Ecole Norm. Sup.}, 29:519--533, 1996.

\bibitem{Fogg:2002}
N.~P. Fogg.
\newblock {\em Substitutions in dynamics, arithmetics and combinatorics},
  volume 1794 of {\em Lecture Notes in Mathematics}.
\newblock Springer-Verlag, Berlin, 2002.
\newblock Edited by V. Berth{\'e}, S. Ferenczi, C. Mauduit and A. Siegel.

\bibitem{Hadamard:1898}
J.~Hadamard.
\newblock Sur la forme des lignes g\'eod\'esiques \`a l'infini et sur les
  g\'eod\'esiques des surfaces r\'egl\'ees du second ordre.
\newblock {\em Bull. Soc. Math. France}, 26:195--216, 1898.

\bibitem{Hollander&Solomyak:2003}
M.~Hollander and B.~Solomyak.
\newblock Two-symbol {P}isot substitutions have pure discrete spectrum.
\newblock {\em Ergod. Th. \& Dynam. Sys.}, 23:533--540, 2003.

\bibitem{Holton&Zamboni:1998}
C.~Holton and L.~Q. Zamboni.
\newblock Geometric realizations of substitutions.
\newblock {\em Bull. Soc. Math. France}, 126:149--179, 1998.

\bibitem{Host:1986}
B.~Host.
\newblock Valeurs propres des syst\`emes dynamiques d\'efinis par des
  substitutions de longueur variable.
\newblock {\em Ergod. Th. \& Dynam. Sys.}, 6:529--530, 1986.

\bibitem{Host:1992}
B.~Host.
\newblock Repr\'esentation g\'eom\'etrique des substitutions sur 2 lettres.
\newblock {\em Un\-pu\-blished manuscript}, 1992.

\bibitem{Kulesza:1995}
J.~Kulesza.
\newblock Zero-dimensional covers of finite-dimensional dynamical systems.
\newblock {\em Ergod. Th. \& Dynam. Sys.}, 15:939--950, 1995.

\bibitem{Morse:1921}
H.~M. Morse.
\newblock Recurrent geodesics on a surface of negative curvature.
\newblock {\em Trans. Amer. Math. Soc.}, 22:84--100, 1921.

\bibitem{Mosse:1992}
B.~Moss{\'e}.
\newblock Puissances de mots et reconnaissabilit\'e des points fixes d'une
  substitution.
\newblock {\em Theoret. Comput. Sci.}, 99:327--334, 1992.

\bibitem{Mosse:1996}
B.~Moss{\'e}.
\newblock Reconnaissabilit\'e des substitutions et complexit\'e des suites
  automatiques.
\newblock {\em Bull. Soc. Math. France}, 124:329--346, 1996.

\bibitem{Queffelec:2010}
M.~Queff{\'e}lec.
\newblock {\em Substitution dynamical systems---spectral analysis}, volume 1294
  of {\em Lecture Notes in Mathematics}.
\newblock Springer-Verlag, Berlin, second edition, 2010.

\bibitem{Rauzy:1982}
G.~Rauzy.
\newblock Nombres alg\'ebriques et substitutions.
\newblock {\em Bull. Soc. Math. France}, 110:147--178, 1982.

\bibitem{Siegel:2003}
A.~Siegel.
\newblock Repr\'esentation des syst\`emes dynamiques substitutifs non
  unimodulaires.
\newblock {\em Ergod. Th. \& Dynam. Sys.}, 23:1247--1273, 2003.

\bibitem{Siegel:2004}
A.~Siegel.
\newblock Pure discrete spectrum dynamical system and periodic tiling
  associated with a substitution.
\newblock {\em Ann. Inst. Fourier}, 54:341--381, 2004.

\bibitem{Sing:2006}
B.~Sing.
\newblock {\em Pisot substitutions and beyond}.
\newblock PhD thesis, Bielefeld University, 2006.

\bibitem{Solomyak:2021pers}
B.~Solomyak.
\newblock Personal communication, 2021.

\end{thebibliography}

\end{document}